\documentclass[12pt,reqno]{amsart} % without "final" showkeys will function

\usepackage{float}

\usepackage{epsfig,amssymb,amsmath,version}
\usepackage{amssymb,version,graphicx,fancybox,mathrsfs}

\usepackage{url,hyperref}
\usepackage{subfigure}
\usepackage{color}
\usepackage{stmaryrd}
\usepackage{multirow}
\usepackage{booktabs,siunitx}
\usepackage{booktabs}
\usepackage{multicol,epstopdf}
\usepackage{xypic}
\usepackage[]{graphicx}
\usepackage[all]{xy}
\usepackage{tikz,ifthen}
\usetikzlibrary{positioning, math, decorations.markings, arrows.meta}
\usetikzlibrary{calc,shapes,cd}
\usetikzlibrary{knots} %, cd,patterns,hobby} %hobby draws ugly paths
\usepgfmodule{decorations}
\allowdisplaybreaks
\tikzset{knotarrow/.pic={ \draw[edge, <-] (0,0) -- +(-.001,0);}}

\tikzset{edge/.style={line width=0.8}}
\tikzset{wall/.style={very thick}}

\tikzset{->-/.style n args={2}{decoration={markings, mark=at position #1 with {\arrow{#2}}}, postaction={decorate}}} %adding end=stealth will make the arrows thicker.

\ExplSyntaxOn
\cs_new_eq:NN \ifstreqF \str_if_eq:nnF
\cs_new_eq:NN \ifstreqTF \str_if_eq:nnTF
\ExplSyntaxOff

\tikzset{-o-/.code 2 args={\ifstreqF{#2}{} %none}
{\ifstreqTF{#2}{>}
   {\pgfkeysalso{decoration={markings,mark=at position #1 with {\arrow[scale=0.8]{#2}}}
                    ,postaction={decorate}}
    }
   {\ifstreqTF{#2}{<}
       {\pgfkeysalso{decoration={markings,mark=at position #1 with {\arrow[scale=0.8]{#2}}}
                    ,postaction={decorate}}
        }
       {\pgfkeysalso{decoration={markings,
                    mark=at position #1 with
                    {\draw[black, fill={#2}] circle[radius=2pt];}}
                    ,postaction={decorate}}
        }
     }
  }}}

\allowdisplaybreaks[4]

\newtheorem{theorem}{Theorem}[section]
\newtheorem{lemma}[theorem]{Lemma}
\newtheorem{definition}[theorem]{Definition}
\newtheorem{corollary}[theorem]{Corollary}
\newtheorem{proposition}[theorem]{Proposition}
\newtheorem{remark}[theorem]{Remark}

\newcommand{\bp}{\begin{proposition}}
\newcommand{\ep}{\end{proposition}}
\newcommand{\bpr}{\begin{proof}}
\newcommand{\epr}{\end{proof}}

\newcommand{\bt}{\begin{theorem}}
\newcommand{\et}{\end{theorem}}

\newcommand{\bl}{\begin{lemma}}
\newcommand{\el}{\end{lemma}}

\newcommand{\bcr}{\begin{corollary}}
\newcommand{\ecr}{\end{corollary}}

\newcommand{\be}{\begin{equation}}
\newcommand{\ee}{\end{equation}}

\newcommand{\bes}{\begin{equation*}}
\newcommand{\ees}{\end{equation*}}

\newcommand{\ba}{\begin{align}}
\newcommand{\ea}{\end{align}}

\newcommand{\bas}{\begin{align*}}
\newcommand{\eas}{\end{align*}}

\graphicspath{{./Figures/} } %{../../pictures}} % indicate the path of figures/pictures

\begin{document}
\bibliographystyle{plain}

\title{Representation-reduced stated skein modules and algebras}
\author{Zhihao Wang}
\address{Zhihao Wang, School of Physical and Mathematical Sciences, Nanyang Technological University, 21 Nanyang Link Singapore 637371}
\email{ZHIHAO003@e.ntu.edu.sg}
\address{University of Groningen, Bernoulli Institute, 9700 AK Groningen, The Netherlands}
\email{wang.zhihao@rug.nl}


 \maketitle

%\tableofcontents

%\tableofcontents{}

\def \cF {\mathcal{F}}
\def \SMQ {\mathscr{S}_{q^{1/2}}(M,\mathcal{N})}
\def \SMQP {\mathscr{S}_{q^{1/2}}(M^{'},\mathcal{N}^{'})}

\def \q {q^{\frac{1}{2}}}

\def \CSM {\mathscr{S}_1(M,\mathcal{N})}

\def \SMN {\mathscr{S}_{q^{1/2}}(M,\mathcal{N})^{(N)}}

\def \S {\mathscr{S}_{q^{1/2}}}

\def \sS {\mathscr{S}}

\def \MN {(M,\mathcal{N})}

\def \sSq {\S(\Sigma)}
\def \sSo {\sS_1(\Sigma)}

\def \bC {\mathbb{C}}

\def \bN {\mathbb{N}}

\def \bZ {\mathbb{Z}}

\def \SZ {Z_{q^{1/2}}(\Sigma)}
\def\sSN {\sSq^{(N)}}
\def \tSZ {\widetilde{{Z_{q^{1/2}}(\Sigma)}}}
\def\tSN {\widetilde{\sSq^{(N)}}}

\begin{abstract}

For any marked three manifold $\MN$ and any quantum parameter $q^{\frac{1}{2}}$ (a nonzero complex number), we use $\mathscr{S}_{q^{1/2}}(M,\mathcal{N})$ to denote the stated skein module of $(M,\mathcal{N})$.
  When $q^{\frac{1}{2}}$ is a root of unity of odd order, the commutative algebra $\mathscr{S}_1(M,\mathcal{N})$ acts on $\mathscr{S}_{q^{1/2}}(M,\mathcal{N})$. For any maximal ideal $\rho$ of $\mathscr{S}_1(M,\mathcal{N})$, define $\mathscr{S}_{q^{1/2}}(M,\mathcal{N})_{\rho} = \mathscr{S}_{q^{1/2}}(M,\mathcal{N})\otimes _{\mathscr{S}_1(M,\mathcal{N})} (\mathscr{S}_1(M,\mathcal{N})/\rho)$.

  We prove the splitting map for $\mathscr{S}_{q^{1/2}}(M,\mathcal{N})$ respects the $\mathscr{S}_1(M,\mathcal{N})$-module structure, so it reduces to the splitting map for
  $\mathscr{S}_{q^{1/2}}(M,\mathcal{N})_{\rho}$. We prove the splitting map for $\mathscr{S}_{q^{1/2}}(M,\mathcal{N})_{\rho}$ is injective if there exists at least one component of $\mathcal{N}$ such that this component and the boundary of the splitting disk belong to the same component of $\partial M$. We also prove the representation-reduced stated skein module of the marked handlebody is an irreducible Azumaya representation of the stated skein algebra of its boundary.

  Let $M$ be an oriented connected closed three manifold. For any positive integer $k$, we use
  $M_{k}$ to denote the marked three manifold obtained from $M$ by removing $k$ open three dimensional balls and adding one marking to each newly created sphere boundary component. We prove $\text{dim}_{\mathbb{C}}
  \mathscr{S}_{q^{1/2}}( M_{k})_{\rho} = 1$ for any maximal ideal $\rho$ of $\mathscr{S}_1(M_k)$.

\end{abstract}

	\def \cN {\mathcal{N}}

\section{Introduction}
A {\bf marked three manifold} is a pair $\MN$, where $M$ is an oriented three manifold, and $\cN$ is a one dimensional submanifold of $\partial M$ consisting of finitely many oriented  open intervals such that there is no intersection between the closure of any two open intervals. The component of $\cN$ is called the {\bf marking} of $\MN$.

For any marked three manifold $\MN$ and any quantum parameter $\q$, the stated skein module $\S\MN$ (refer to section \ref{stated} for the definition) is a vector space over the complex field $\bC$. If $\MN$ is the thickening of a marked surface (refer to section \ref{stated} for the definition), then $\S\MN$ has an algebra structure (in this case, we call the stated module as the stated skein algebra).
To simplify the notation, we will use $\mathscr{S}$ to denote $\mathscr{S}_{q^{1/2}}$.

\def \SMN {\mathscr{S}(M,\mathcal{N})^{(N)}}
\def \SMQ {\mathscr{S}(M,\mathcal{N})}
\def \SMQP {\mathscr{S}(M^{'},\mathcal{N}^{'})}
\def \S {\mathscr{S}}
\def \max {\text{MaxSpec}(\sS_1\MN)}

When $\q$ is a root of unity of odd order, there exists a linear map 
$$\cF:\sS_{1}\MN\rightarrow \S\MN,$$
called the Frobenius map \cite{bloomquist2020chebyshev,bonahon2016representations}.
When $\MN$ is the thickening of a marked surface, the Frobenius map $\cF$
is an algebra embedding and $\text{Im}\cF$ is contained in the center of $\S\MN$ \cite{bloomquist2020chebyshev,bonahon2016representations}.

It is well-known that $\sS_1\MN$ has a commutative algebra structure, defined by taking the disjoint union of stated $\MN$-tangles. We use
$\text{MaxSpec}(\sS_1\MN)$ to denote the set of all maximal ideals of $\sS_1\MN$.  Then $\text{MaxSpec}(\sS_1\MN)$ is actually isomorphic to a representation variety when every component of $M$ contains at least one marking, and is isomorphic to a character variety when $\mathcal{N} = \emptyset$ \cite{bullock1997rings,wang2023stated}. 

When $\q$ is a root of unity of odd order,
 the commutative algebra $\sS_1\MN$ has an action on $\S\MN$ via the Frobenius map, see subsection \ref{sub3.1} for the definition of this action.  Then for any $\rho\in \max$, define $$\S\MN_{\rho}
 =\S\MN\otimes_{\sS_1\MN} (\sS_1\MN/\rho)\simeq \S\MN/\rho\cdot\S\MN.$$
 The parallel definition for this quotient vector space for the skein module (non-stated case) is 
defined in \cite{frohman2023sliced}. 
 The author proved $\S\MN$ is finitely generated over $\sS_1\MN$ when $M$ is compact \cite{wang2023finiteness}. This implies, for any $\rho\in \max,$ the vector space $\S\MN_{\rho}$ is finite dimensional when $M$ is compact \cite{wang2023finiteness}.

When $\MN$ is the thickening of a marked surface $(\Sigma,\mathcal{P})$, it is easy to show $\S(\Sigma,\mathcal{P})_{\rho}$ is a quotient algebra of $\S(\Sigma,\mathcal{P})$ for any $\rho\in\text{MaxSpec}(\sS_1(\Sigma,\mathcal{P}))$, where  $\S(\Sigma,\mathcal{P})_{\rho}$ denotes $\S\MN_{\rho}$ and $\S(\Sigma,\mathcal{P})$ denotes $\S\MN$. This quotient algebra  is very important to understand the representation theory of the stated skein algebra.  Actually any irreducible representation of $\S(\Sigma,\mathcal{P})$ reduces to an irreducible representation of $\S(\Sigma,\mathcal{P})_{\rho}$ for some $\rho\in\text{MaxSpec}(\sS_1(\Sigma,\mathcal{P}))$. 

%Let $\MN$ be a marked three manifold, and let $(D,\beta)$ be a pair, where 
%$D$, called the splitting disk, is a properly embedded disk and $\beta$ is an embedded oriented open interval in $D$. Suppose $U(D)$ is an open regular neighborhood of $D$ such that $U(D)$ is isomorphic to $D\times (0,1)$ and $\partial U(D) = \partial D\times (0,1)$.  Let $M^{'} = M\setminus U(D)$. Then there exists a projecture $\text{pr}:M^{'}\rightarrow M$. Suppose $\text{pr}^{-1}(\beta) =\beta^{'}\cup\beta^{''}$, where both $\beta^{'}$ and $\beta^{''}$ are oriented open intervals in $\partial M^{'}$. Define $\text{Cut}_{(D,\beta)}\MN = (M^{'},\mathcal{N}^{'})$, where $N^{'} = N\cup\beta^{'}\cup\beta^{''}$.

Let $\MN$ be a marked three manifold. For any properply embedded disk $D$ in $M$ and any embedded oriented open interval $u\subset D$, 
we use $\text{Cut}_{(D,u)}\MN$ to denote the marked three manifold obtained from $\MN$ by cutting along $(D,u)$, please refer to subsection \ref{a} for more details.
There exists a linear map 
\begin{equation}\label{eqs}
\Theta:\S\MN\rightarrow \S(\text{Cut}_{(D,u)}\MN)
\end{equation}
 called the splitting map.
When $\q = 1$, we have $\Theta:\sS_1\MN\rightarrow \sS_1(\text{Cut}_{(D,u)}\MN)$ is an algebra homorphism. Then it induces a map $$\Theta^{*}:\text{MaxSpec}(\sS_1(\text{Cut}_{(D,u)}\MN))\rightarrow \max.$$
We prove, for any $\rho\in \text{MaxSpec}(\sS_1(\text{Cut}_{(D,u)}\MN))$, the splitting map in equation \eqref{eqs} induces a linear map
$$\Theta_{\rho}:\S\MN_{\Theta^{*}(\rho)}\rightarrow \S(\text{Cut}_{(D,u)}\MN)_{\rho}.$$

\begin{theorem}\label{1.1}
	Let $\MN$ be a marked three manifold,  let $D$ be a properly embedded  disk in $M$, and let $u$ be an embedded oriented open interval in $D$. Suppose the component $V$ of $\partial M$  contains  $\partial D$ and $V\cap \cN\neq\emptyset$. Then, for any $\rho\in \text{MaxSpec}(\sS_1(\text{Cut}_{(D,u)}\MN))$, we have
	$\Theta_{\rho}$
	is injective.
\end{theorem}

We use $H_g$ to denote the genus $g$ handlebody. For any positive integer $k$, we use $H_{g,k}$ to denote the marked three manifold obtained from $H_g$ by adding $k$ markings on $\partial H_g$. Then $H_{g,k}$ is defined up to isomorphism.
We use $\Sigma_g$ to denote the closed surface of genus $g$. For any positive integer $k$, we use $\Sigma_{g,k}$ to denote the marked surface  obtained from $\Sigma_g$ by removing $k$ open disks and equipping each newly created boundary component with one marked point. Then the stated skein algebra $\S(\Sigma_{g,k})$ acts on $\S(H_{g,k})$, please refer to section \ref{5} for more details.

\begin{theorem}\label{1.2}
	Let $k$ be a positive integer.
	For any $\rho\in \text{MaxSpec}(\sS_1(H_{g,k}))$, we have $\S(H_{g,k})_{\rho}$ is an irreducible representation of $\S(\Sigma_{g,k})$. Meanwhile, it is an Azumaya  representation of $\S(\Sigma_{g,k})$, please refer to section \ref{5} for the definition of the Azumaya  representation.
\end{theorem}

%A parallel result for Theorem \ref{1.2} for the skein case (that is non-stated case) is proved in \cite{frohman2023sliced}. For the non-stated case, they have some restrictions for $\rho$. Theorem \ref{1.2} holds for all $\rho\in \text{MaxSpec}(\sS_1(H_{g,k}))$.

Let $M$ be an oriented connected closed three manifold. For any positive integer $k$, we use 
$M_{k}$ to denote the marked three manifold obtained from $M$ by removing $k$ open three dimensional balls and adding one marking to each newly created sphere boundary component.

\begin{theorem}\label{1.3}
	Let $M$ be any oriented connected closed three manifold, and let $k$ be any positive integer. For any $\rho\in\text{MaxSpec}(\sS_1(M_{k}))$, we have $\S(M_k)_{\rho} \simeq \bC$ (as $\bC$-vector spaces).
	
\end{theorem}

Parallel results for Theorems \ref{1.2} and \ref{1.3} for the skein case (that is the non-stated case) are proved in \cite{frohman2023sliced}. For the non-stated case, they have some restrictions for $\rho$. Theorems \ref{1.2} and \ref{1.3} hold for any $\rho\in \text{MaxSpec}(\sS_1(H_{g,k}))$ or $\rho\in\text{MaxSpec}(\sS_1(M_{k}))$.

{\bf Plan of this paper.}  In section \ref{stated}, we introduce the definition of stated skein modules and algebras.
In section \ref{sec3}, we introduce the Forbenius map and define the representation-reduced stated skein modules and algebras.  In section \ref{666}, we prove Theorem \ref{1.1}. In section \ref{5}, we prove Theorems \ref{1.2} and \ref{1.3}. 

{\bf 
Acknowledgements}:  The research is supported by the NTU  research scholarship from the Nanyang Technological University (Singapore) and the PhD scholarship from the University of Groningen (The Netherlands).

\section{Stated skein modules and algebras}\label{stated}

For a marked three manifold $\MN$, a properly embedded one dimensional submanifold $\alpha$ of $M$  is called an
{\bf $\MN$-tangle} if $\partial\alpha\subset\mathcal{N}$ and $\alpha$ is equipped with a framing  such that  framings at $\partial\alpha$ respect to velocity vectors of $\mathcal{N}$.
If there is a map $s:\partial\alpha\rightarrow \{-,+\}$, then we call $\alpha$ a
 {\bf stated $\MN$-tangle}.

For a marked three manifold $\MN$, 
we use $\text{Tangle}\MN$ to denote the vector space over $\bC$ with all isotopy classes of stated $\MN$-tangles as a basis.
Then the stated skein module $\SMQ$ of $\MN$ is  $\text{Tangle}\MN$ quotient the following relations:
\begin{equation}\label{cross}
\raisebox{-.20in}{
\begin{tikzpicture}%[dline /. style ={line width =2pt}]
\filldraw[draw=white,fill=gray!20] (-0,-0.2) rectangle (1, 1.2);%gray part 
\draw [line width =1pt](0.6,0.6)--(1,1);
\draw [line width =1pt](0.6,0.4)--(1,0);
\draw[line width =1pt] (0,0)--(0.4,0.4);
\draw[line width =1pt] (0,1)--(0.4,0.6);
\draw[line width =1pt] (0.6,0.6)--(0.4,0.4);% negative crossing
%\draw[line width =1pt] (0.4,0.6)--(0.6,0.4);%positive  crossing
\end{tikzpicture}
}=
q
\raisebox{-.20in}{
\begin{tikzpicture}%[dline /. style ={line width =2pt}]
\filldraw[draw=white,fill=gray!20] (-0,-0.2) rectangle (1, 1.2);%gray part 
\draw [line width =1pt](0.6,0.6)--(1,1);
\draw [line width =1pt](0.6,0.4)--(1,0);
\draw[line width =1pt] (0,0)--(0.4,0.4);
\draw[line width =1pt] (0,1)--(0.4,0.6);
\draw[line width =1pt] (0.6,0.62)--(0.6,0.38);% negative crossing
\draw[line width =1pt] (0.4,0.38)--(0.4,0.62);
%\draw[line width =1pt] (0.4,0.6)--(0.6,0.4);%positive  crossing
\end{tikzpicture}
}
+
 q^{-1}
\raisebox{-.20in}{
\begin{tikzpicture}%[dline /. style ={line width =2pt}]
\filldraw[draw=white,fill=gray!20] (-0,-0.2) rectangle (1, 1.2);%gray part 
\draw [line width =1pt](0.6,0.6)--(1,1);
\draw [line width =1pt](0.6,0.4)--(1,0);
\draw[line width =1pt] (0,0)--(0.4,0.4);
\draw[line width =1pt] (0,1)--(0.4,0.6);
\draw[line width =1pt] (0.62,0.6)--(0.38,0.6);% negative crossing
\draw[line width =1pt] (0.62,0.4)--(0.38,0.4);
%\draw[line width =1pt] (0.4,0.6)--(0.6,0.4);%positive  crossing
\end{tikzpicture}
} 
\end{equation}
\begin{equation}\label{unknot}
\raisebox{-.15in}{
\begin{tikzpicture}%[dline /. style ={line width =2pt}]
\filldraw[draw=white,fill=gray!20] (-0,-0) rectangle (1, 1);%gray part 
\draw [line width =1pt] (0.5,0.5) circle (0.3);
\end{tikzpicture}
}=-(q^2+q^{-2})
\raisebox{-.15in}{
\begin{tikzpicture}%[dline /. style ={line width =2pt}]
\filldraw[draw=white,fill=gray!20] (-0,-0) rectangle (1, 1);%gray part 
%\draw (0.5,0.5) circle (0.4);
\end{tikzpicture}
}
\end{equation}
\begin{equation}\label{arc}
\raisebox{-.26in}{
\begin{tikzpicture}%[dline /. style ={line width =2pt}]
\filldraw[draw=white,fill=gray!20] (-0,-0) rectangle (1, 1);%gray part 
\draw [line width =1pt]  (0.5 ,0) arc (-90:225:0.3 and 0.35);
\filldraw[draw=black,fill=black] (0.5,-0) circle (0.09);
\node at (0.3,-0.15) {$-$};
\node at (0.7,-0.15) {\small $+$};
\end{tikzpicture}
}=q^{-\frac{1}{2}}
\raisebox{-.15in}{
\begin{tikzpicture}%[dline /. style ={line width =2pt}]
\filldraw[draw=white,fill=gray!20] (-0,-0) rectangle (1, 1);%gray part 
%\draw [line width =1pt]  (0.5 ,0) arc (-90:225:0.3 and 0.35);
\filldraw[draw=black,fill=black] (0.5,-0) circle (0.09);
\end{tikzpicture}
},\;
\raisebox{-.26in}{
\begin{tikzpicture}%[dline /. style ={line width =2pt}]
\filldraw[draw=white,fill=gray!20] (-0,-0) rectangle (1, 1);%gray part 
\draw [line width =1pt]  (0.5 ,0) arc (-90:225:0.3 and 0.35);
\filldraw[draw=black,fill=black] (0.5,-0) circle (0.09);
\node at (0.3,-0.15) {$+$};
\node at (0.7,-0.15) {\small $+$};
\end{tikzpicture}
}=
\raisebox{-.26in}{
\begin{tikzpicture}%[dline /. style ={line width =2pt}]
\filldraw[draw=white,fill=gray!20] (-0,-0) rectangle (1, 1);%gray part 
\draw [line width =1pt]  (0.5 ,0) arc (-90:225:0.3 and 0.35);
\filldraw[draw=black,fill=black] (0.5,-0) circle (0.09);
\node at (0.3,-0.15) {$-$};
\node at (0.7,-0.15) {\small $-$};
\end{tikzpicture}
} =0
\end{equation}
\begin{equation}\label{hight}
\raisebox{-.26in}{
\begin{tikzpicture}%[dline /. style ={line width =2pt}]
\filldraw[draw=white,fill=gray!20] (-0.2,-0) rectangle (1.2, 1);%gray part 
\draw [line width =1pt](0,1)--(0.4,0.2);
\draw [line width =1pt](1,1)--(0.6,0.2);
\draw [line width =1pt](0.5,0)--(0.6,0.2);
\filldraw[draw=black,fill=black] (0.5,-0) circle (0.09);
\node at (0.3,-0.15) {$+$};
\node at (0.7,-0.15) {\small $-$};
\end{tikzpicture}
}=q^{2}
\raisebox{-.26in}{
\begin{tikzpicture}%[dline /. style ={line width =2pt}]
\filldraw[draw=white,fill=gray!20] (-0.2,-0) rectangle (1.2, 1);%gray part 
\draw [line width =1pt](0,1)--(0.4,0.2);
\draw [line width =1pt](1,1)--(0.6,0.2);
\draw [line width =1pt](0.5,0)--(0.6,0.2);
\filldraw[draw=black,fill=black] (0.5,-0) circle (0.09);
\node at (0.3,-0.15) {$-$};
\node at (0.7,-0.15) {\small $+$};
\end{tikzpicture}
} + q^{-\frac{1}{2}}
\raisebox{-.16in}{
\begin{tikzpicture}%[dline /. style ={line width =2pt}]
\filldraw[draw=white,fill=gray!20] (-0.2,-0) rectangle (1.2, 1);%gray part 
\draw [line width =1pt](0,1)--(0.4,0.2);
\draw [line width =1pt](1,1)--(0.6,0.2);
\draw [line width =1pt](0.38,0.2)--(0.62,0.2);
\filldraw[draw=black,fill=black] (0.5,-0) circle (0.09);
\end{tikzpicture}
}
\end{equation}
\def \cN {\mathcal{N}}
where each black dot represents an oriented interval in $\cN$ with the orientation pointing towards readers, each gray square represents a projection of an embedded cube in $M$, the black lines are parts of stated $\MN$-tangles, and in each equation the parts of $\MN$-tangles outside the gray squares are identical.  Please refer to \cite{bloomquist2020chebyshev, le2018triangular} for the detailed explanation.

%An arc  in 

If $\cN=\emptyset$, there are only equations \eqref{cross} and \eqref{unknot} involved. In this case, the stated skein module is the (Kauffman bracket) skein module \cite{przytycki2006skein,turaev1988conway}.

Let $(M_1,\mathcal{N}_1)$ and $(M_2,\mathcal{N}_2)$ be two marked three manifolds. An embedding $f:M_1\rightarrow M_2$ is called an embedding from 
 $(M_1,\mathcal{N}_1)$ to $(M_2,\mathcal{N}_2)$ if $f(\mathcal{N}_1)\subset\mathcal{N}_2$ and $f$ preserves the orientations between $\mathcal{N}_1$ and $\mathcal{N}_2$. Clearly $f$ induces a linear map $f_{*}:\S(M_1,\mathcal{N}_1)\rightarrow \S(M_2,\mathcal{N}_2)$. If there exist embedding $g_1: (M_1,\mathcal{N}_1)\rightarrow (M_2,\mathcal{N}_2)$ and embedding $g_2:(M_2,\mathcal{N}_2)\rightarrow (M_1,\mathcal{N}_1)$ such that 
$g_2\circ g_1 = Id_{(M_1,\mathcal{N}_1)}$ and $g_1\circ g_2 = Id_{(M_2,\mathcal{N}_2)}$, we say $(M_1,\mathcal{N}_1)$ and $(M_2,\mathcal{N}_2)$ are isomorphic to each other.

A {\bf marked surface} is a pair $(\Sigma,\mathcal{P})$, where 
$\Sigma$ is a compact surface and $\mathcal{P}$ is a set of finite points in $\partial \Sigma$, called {\bf marked points}.

The {\bf bigon} $\mathcal{B}$ is  the 2-dimensional disk with two marked points on the boundary. The {\bf marked triangle} $\mathfrak T$ is the 2-dimensional disk with three marked points on the boundary.
 %The {\bf monogon} is obtained from the 2-dimensional disk by removing one puncture on the boundary. 
%Let $\overline{\Sigma}$ be an oriented compact surface. A {\bf punctured bordered surface} (or  {\bf pb surface}) $\Sigma$ is obtained from $\overline{\Sigma}$ by removing finite points, which are called punctures, such that every boundary component of $\overline{\Sigma}$ contains at least one puncture. The punctures, contained in the interior of $\overline{\Sigma}$, are called {\bf interior punctures}.
%We define $r(\Sigma) = -\chi({\Sigma})+\sharp\partial\Sigma$ where 
%$\chi(\Sigma)$ is the Euler characteristic of $\Sigma$ and $\sharp\partial\Sigma$ is the number of boundary components of $\Sigma$.

%The {\bf bigon} is obtained from the 2-dimensional disk by removing two punctures on the boundary. The {\bf monogon} is obtained from the 2-dimensional disk by removing one puncture on the boundary. 
%We will use $\mathcal{D}$ and $\mathcal{M}$ to denote the bigon and the monogon respectively.

For a marked surface $(\Sigma,\mathcal{P})$, we set $M=\Sigma\times [0,1]$, and set $\cN=\mathcal{P}\times (0,1)$ where the orientation of every component of $\cN$ is given by the positive direction of $(0,1)$.
Then $\MN$ is a marked three manifold. We call $\MN$ the thickening of $(\Sigma,\mathcal{P})$, and define
$\S(\Sigma,\mathcal{P})=\S\MN$.
Then $\S(\Sigma,\mathcal{P})$ has an algebra structure given by staking stated  tangles, that is, for any two stated  tangles $\alpha,\beta$, the product $\alpha\beta$ is defined to be staking $\alpha$ above $\beta$. We call $\S(\Sigma,\mathcal{P})$ the stated skein algebra of $(\Sigma,\mathcal{P})$. 
When $\mathcal{P}=\emptyset$, the algebra $\S(\Sigma,\mathcal{P})$ is the (Kauffman bracket) skein algebra.

Let $(\Sigma_1,\mathcal{P}_1)$ and $(\Sigma_2,\mathcal{P}_2)$ be two marked surfaces. An embedding $f:\Sigma_1\rightarrow \Sigma_2$ is called an embedding from $(\Sigma_1,\mathcal{P}_1)$ to $(\Sigma_2,\mathcal{P}_2)$ if $f(\mathcal{P}_{1})\subset \mathcal{P}_{2}$. Clearly $f$ induces an algebra homomorphism $f_*:\S(\Sigma_1,\mathcal{P}_1)\rightarrow \S(\Sigma_2,\mathcal{P}_2)$. Similarly as the marked three manifold, we define two marked surfaces are isomorphic to each other, if there exist two embeddings respectively from $(\Sigma_1,\mathcal{P}_1)$ to
$(\Sigma_2,\mathcal{P}_2)$ and from $(\Sigma_2,\mathcal{P}_2)$ to $(\Sigma_1,\mathcal{P}_1)$ such that they are inverse to each other.

From now on, we will always assume $\q$ is a root of unity of odd order $N$.

%For any two pb surfaces $\Sigma_1,\Sigma_2$, obviously we have $\S(\Sigma_1\cup\Sigma_2)\simeq \S(\Sigma_1)\otimes\S(\Sigma_2)$ as algebras. From now on, we  assume all the pb surfaces, involved in this paper, are connected.

%We call the orientation of $\partial\Sigma$, induced by the orientation of $\Sigma$, the {\bf positive orientation} of $\partial\Sigma$, and call the orientation of $\partial\Sigma$ opposite to the positive orientation as the {\bf negative orientation.}
%
%Let $h:\Sigma_1\rightarrow \Sigma_2$ be a proper embedding for two pb surfaces $\Sigma_1,\Sigma_2$.
%It is possible that $h$ maps more than one boundary component of $\Sigma_1$ into one boundary component of  
%$\Sigma_2$. For each boundary component $b$ of $\Sigma_2$, we give a linear order on the set of boundary components of $\Sigma_1$ that are mapped into $b$ by $h$, and call $h$ an {\bf ordered proper embedding} from 
%$\Sigma_1$ to $\Sigma_2$. If the linear order, corresponding to each boundary component $b$ of $\Sigma_2$, is induced by the positive (respectively negative) orientation of $\partial\Sigma_2$, we say $h$ is {\bf positively ordered} (respectively {\bf negatively ordered}). An ordered proper embedding $h:\Sigma_1\rightarrow \Sigma_2$ induces a linear map $h_{*}:\S(\Sigma_1)\rightarrow \S(\Sigma_2)$ \cite{costantino2022stated1}.

\section{Representation-reduced stated skein modules and algebras}\label{sec3}

In order to define the Frobenius map, first we have to introduce Chebyshev polynomials, which are defined by the following 
recurrenc relation:
\begin{equation}\label{cheby}
	T_k(x) = xT_{k-1}(x) - T_{k-2}(x).
\end{equation}
We set $T_0(x) = 2, T_1(x) = x$. Using the recurrenc relation \eqref{cheby}, we get a sequence of polynomials
$\{T_n(x)\}_{n\in\mathbb{N}}$, which are called
Chebyshev polynomials of the first kind (here we use $\mathbb{N}$ to denote the set of non-negative integers).

\subsection{Frobenius map}\label{sub3.1}
Let $(M,\mathcal{N})$ be a marked three manifold, and $\alpha$ be a framed knot or stated framed arc in
$\MN$. For any $n\in \bN$, we use $\alpha^{(n)}$ to denote a new stated $\MN$-tangle obtained   by threading $\alpha$ to $n$ parallel copies along the framing direction.
For a polynomial $Q(x)=\sum_{0\leq t\leq n} k_t x^t$,  define
$$\alpha^{[Q]} = \sum_{0\leq t\leq n} k_t \alpha^{(t)}\in\SMQ.$$

\def \cF {\mathcal{F}}

There is a linear map 
$\cF:\sS_{1}\MN\rightarrow \SMQ,$
called the Frobenius map \cite{bloomquist2020chebyshev,bonahon2016representations}.
Let $\alpha$ be any stated $(M,\mathcal{N})$-tangle, suppose 
$\alpha = K_1\cup\cdots\cup K_m \cup C_1\cup\cdots\cup C_n$
where $K_i,1\leq i\leq m,$ are framed knots and $C_j,1\leq j\leq n,$ are stated framed arcs.
Then 
\begin{equation}\label{Frobe}
	\cF(\alpha) =K_1^{[T_N]}\cup\cdots\cup K_m^{[T_N]}\cup C_1^{(N)}\cup\cdots\cup C_n^{(N)}.
\end{equation}

\def \Im {\text{Im}}

\def \Zq {Z_{q^{1/2}}(\Sigma)}

Let $\alpha$ be any stated $(M,\mathcal{N})$-tangle, then $\cF(\alpha)$ is transparent \cite{bloomquist2020chebyshev,bonahon2016representations,wang2023stated}, in a sense that we have the following two relations in $\SMQ$:
\begin{equation}\label{well_de}
	\raisebox{-.20in}{
		\begin{tikzpicture}%[dline /. style ={line width =2pt}]
			\filldraw[draw=white,fill=gray!20] (-0,-0.2) rectangle (1, 1.2);%gray part 
			\draw [line width =1pt](0.6,0.6)--(1,1);
			\draw [line width =1pt](0.6,0.4)--(1,0);
			\draw[line width =1pt] (0,0)--(0.4,0.4);
			\draw[line width =1pt] (0,1)--(0.4,0.6);
			\draw[line width =1pt] (0.6,0.6)--(0.4,0.4);% negative crossing
			%\draw[line width =1pt] (0.4,0.6)--(0.6,0.4);%positive  crossing
			\node [right] at (1,1) {$\cF(\alpha)$};
		\end{tikzpicture}
	}=
	\raisebox{-.20in}{
		\begin{tikzpicture}%[dline /. style ={line width =2pt}]
			\filldraw[draw=white,fill=gray!20] (-0,-0.2) rectangle (1, 1.2);%gray part 
			\draw [line width =1pt](0.6,0.6)--(1,1);
			\draw [line width =1pt](0.6,0.4)--(1,0);
			\draw[line width =1pt] (0,0)--(0.4,0.4);
			\draw[line width =1pt] (0,1)--(0.4,0.6);
			%\draw[line width =1pt] (0.6,0.6)--(0.4,0.4);% negative crossing
			\draw[line width =1pt] (0.4,0.6)--(0.6,0.4);%positive  crossing
			\node [right] at (1,1) {$\cF(\alpha)$};
		\end{tikzpicture}
	},
	\raisebox{-.35in}{
		\begin{tikzpicture}%[dline /. style ={line width =2pt}]
			\filldraw[draw=white,fill=gray!20] (-0.2,-0) rectangle (1.2, 1);%gray part 
			\draw [line width =1pt](0,1)--(0.4,0.2);
			\draw [line width =1pt](1,1)--(0.6,0.2);
			\draw [line width =1pt](0.5,0)--(0.6,0.2);
			\filldraw[draw=black,fill=black] (0.5,-0) circle (0.09);
			\node at (0.3,-0.15) {$i$};
			\node at (0.7,-0.15) {\small $j$};
			\node [above] at (1,1) {$\cF(\alpha)$};
		\end{tikzpicture}
	}=
	\raisebox{-.35in}{
		\begin{tikzpicture}%[dline /. style ={line width =2pt}]
			\filldraw[draw=white,fill=gray!20] (-0.2,-0) rectangle (1.2, 1);%gray part 
			\draw [line width =1pt](0,1)--(0.4,0.2);
			\draw [line width =1pt](1,1)--(0.6,0.2);
			\draw [line width =1pt](0.5,0)--(0.4,0.2);
			\filldraw[draw=black,fill=black] (0.5,-0) circle (0.09);
			\node at (0.3,-0.15) {$i$};
			\node at (0.7,-0.15) {\small $j$};
			\node [above] at (1,1) {$\cF(\alpha)$};
		\end{tikzpicture}
	} 
\end{equation}
where $i,j= -,+$. 

For any marked three manifold $\MN$, the stated skein module $\sS_1\MN$ has a commutative algebra structure, defined by taking the disjoint union of $\MN$-stated tangles. When $\q$ is a root of unity of odd order,
 there is an action of $\sS_1\MN$ on $\S\MN$, defined by, for any two disjoint stated tangles $\alpha,\beta$, 
$\alpha\cdot\beta = \cF(\alpha)\cup \beta$.

Let $(\Sigma,\mathcal{P})$ be a marked surface, then $\cF:\sS_1(\Sigma,\mathcal{P})\rightarrow \S(\Sigma,\mathcal{P})$ is injective and 
$\Im \cF$ is contained in the center of $\S(\Sigma,\mathcal{P})$ \cite{bloomquist2020chebyshev,bonahon2016representations}.
For any $x\in \sS_1(\Sigma,\mathcal{P})$ and $y\in \S(\Sigma,\mathcal{P})$, the action $x\cdot y$ will be
$\cF(x)y$.
%The action of $\sS_1(\Sigma)$ on $\sSq$ is just 

\def \otau {\overline{\tau}}
\def \Trq {\mathcal{T}_{q^{1/2},\tau}(\Sigma)}
\def \Trp {\mathcal{T}_{q^{1/2},\tau}^{+}(\Sigma)}
\def \Tro {\mathcal{T}_{1,\tau}(\Sigma)}
\def \Tr {\text{Tr}}
\def \mno {\sS_{1}(M,\mathcal{N})}
\def \max {\text{MaxSpec}(\mno)}

\def \Max {\text{MaxSpec}(\sSo)}
\subsection{Representation-reduced stated skein modules and algebras}
Let $(M,\mathcal{N})$ be any marked three manifold. Recall that we use $\text{MaxSpec}(\mno)$
to denote the set of all maximal ideals of $\mno$. 
Note that we can also regard $\text{MaxSpec}(\mno)$ as the set of algebra homomorphisms from $\mno$ to $\mathbb C$.

%Then $\max$ is actually isomorphic to a representation variety when every component of $M$ contains at least one marking and is isomorphic to a character variety when $\mathcal{N} = \emptyset$ \cite{bullock1997rings,wang2023stated}.
%The set of maximal ideals of $\sSo$, which will be denoted as $\text{MaxSpec}(\sSo)$, is an algebraic variety.
%$\text{MaxSpec}(\sSo)$ is actually isomorphic to the representation variety of a fundamental groupoid, associated to $\Sigma$ \cite{costantino2022stated1,wang2023stated}.

For any element $\rho\in\max$, the commutative algebra
$\mno$ has an action on $\bC$ induced by $\rho,$ that is, for $x\in\mno,k\in \bC$, $x\cdot k = \rho(x)k$. Recall that
$\mno$ also has an action on $\S\MN$.
Then we define $$\S\MN_{\rho} = \S\MN\otimes_{\mno}^{\rho}\bC\simeq \S\MN/\text{Ker}(\rho)\cdot\S\MN,$$ where the superscript for $\otimes_{\sSo}^{\rho}$ is to imply the action of 
$\mno$ on $\bC$ is induced by $\rho$. We call $\S\MN_{\rho}$   the representation-reduced stated skein module of $\MN$ related to $\rho$, or just, the representation-reduced stated skein module when $\MN$ and $\rho$ are clear.

\begin{remark}
Frohman-Kania-Bartoszynska-L{\^e} give the parallel definition for the skein case (that is the non-stated case) as the representation-reduced stated skein module \cite{frohman2023sliced}. They call this quotient module as the "character reduced skein module". The reason why they use "character" is because $\max$ is a character variety when $\cN=\emptyset$ (that is for the non-stated case). For the stated case, we call this quotient module as "representation-reduced stated skein module"  bacause $\max$ is a representation variety when every component of $M$ contains at least one marking (that is the stated case).
\end{remark}

When the marked three manifold $\MN$ is the thickening of a marked surface $(\Sigma,\mathcal{P})$ and $\rho\in\text{MaxSpec}(\sS_1(\Sigma,\mathcal{P}))$, we use $\S(\Sigma,\mathcal{P})_{\rho}$ to denote $\S\MN_{\rho}$. Then 
 $\S(\Sigma,\mathcal{P})_{\rho}$ is a quotient algebra of $\S(\Sigma,\mathcal{P})$.
 
 \def \cN {\mathcal{N}}
 
 \begin{definition}
 	Let $(M_1,\cN_1)$ and $(M_2,\cN_2)$ be two marked three manifolds.
 Recall that $q^{\frac{1}{2}}$ is a root of unity of odd order $N$. 
 Suppose for each such a $q^{\frac{1}{2}}$, there exists a linear map
 $\varphi_{q^{1/2}}: \S(M_1,\cN_1)\rightarrow \S(M_2,\cN_2)$. We can omit the subscript for $\varphi_{q^{1/2}}$ when there is no confusion. We  say
 $\varphi: \S(M_1,\cN_1)\rightarrow \S(M_2,\cN_2)$ behaves well with respect to the Frobenius map if it satisfies the following conditions:
 
 (1) If ${q^{1/2}}=1$, then $\varphi: \sS_1(M_1,\cN_1)\rightarrow \sS_1(M_2,\cN_2)$ is an algebra homomorphism. %Especially $f_*$ induces a map
 %\begin{align*} f^{*}:\text{MaxSpec}(\sS_1(M_2,\cN_2))\rightarrow \text{MaxSpec}(\sS_1(M_1,\cN_1)),\;
 %	\rho\mapsto \rho\circ f_*,
% \end{align*}
% where $\rho \in \text{MaxSpec}(\sS_1(M_2,\cN_2))$ is regarded as an algebra homomorphism from $\sS_1(M_2,\cN_2)$ to $\mathbb{C}$.
 
 (2) For each  $q^{\frac{1}{2}}$, we have the following commutative diagram:
 \begin{equation*}
 	\begin{tikzcd}
 		\sS_1(M_1,\cN_1)  \arrow[r, "\varphi_1"]
 		\arrow[d, "\cF"]  
 		&  \sS_1 (M_2,\mathcal{N}_2) \arrow[d, "\cF"] \\
 		\S(M_1,\cN_1)  \arrow[r, "\varphi_{q^{1/2}}"] 
 		&  \S(M_2,\cN_2)\\
 	\end{tikzcd}.
 \end{equation*}
 
 (3) For each $q^{\frac{1}{2}}$, we have $\varphi_{q^{1/2}}: \S(M_1,\cN_1)\rightarrow \S(M_2,\cN_2)$ preserves module structures, in a sense that,
 $\varphi_{q^{1/2}}(\alpha\cdot \beta) = \varphi_1(\alpha)\cdot \varphi_{q^{1/2}}(\beta)$ for any $\alpha\in\sS_1(M_1,\cN_1),\beta\in\S(M_1,\cN_1)$.
 
 \end{definition}
 
  \def \cN {\mathcal{N}}

 \begin{lemma}\label{con}
 	Suppose $\varphi: \S(M_1,\cN_1)\rightarrow \S(M_2,\cN_2)$ behaves well with respect to the Frobenius map. We know the algebra homomorphism  $\varphi_1: \sS_1(M_1,\cN_1)\rightarrow \sS_1(M_2,\cN_2)$ induces a map: $\varphi^{*}:\text{MaxSpec}(\sS_1(M_2,\cN_2))\rightarrow \text{MaxSpec}(\sS_1(M_1,\cN_1))$.
 Then,	 for each $q^{\frac{1}{2}}$ and $\rho\in \text{MaxSpec}(\sS_1(M_2,\cN_2))$, the linear map $$\varphi_{q^{1/2}}:\S(M_1,\cN_1)\rightarrow \S(M_2,\cN_2)$$ induces a linear map $$\varphi_{\rho}: \S(M_1,\cN_1)_{\varphi^{*}(\rho)}\rightarrow \S(M_2,\cN_2)_{\rho}$$ defined by sending $x\otimes k$ to $\varphi(x)\otimes k$, where $x\in\S(M_1,\cN_1)$ and $k\in\mathbb{C}$.
 \end{lemma}

 \begin{proof}

 	For each  $q^{\frac{1}{2}}$ and $\rho\in \text{MaxSpec}(\sS_1(M_2,\cN_2))$, define 
 	$$\overline{\varphi_{\rho}}:\S(M_1,\cN_1)\times \mathbb{C}\rightarrow \S(M_2,\cN_2)_{\rho},\; (x,k)\rightarrow \varphi_{q^{1/2}}(x)\otimes k.$$ Clearly $\overline{\varphi_{\rho}}$ is bilinear. For any $x\in\S(M_1,\cN_1),\, y\in \sS_1(M_1,\cN_1)$ and $k\in\mathbb{C}$, we have 
 	\begin{align*}
 	&\overline{\varphi_{\rho}}(x\cdot y, k) = \varphi_{q^{1/2}}(x\cdot y) \otimes k = \varphi_{q^{1/2}}(x)\cdot \varphi_1( y)\otimes k\\
 	= &\varphi_{q^{1/2}}(x)\otimes \varphi_1( y)\cdot k
 	=\varphi_{q^{1/2}}(x)\otimes \rho(\varphi_1( y)) k
 	= \overline{\varphi_{\rho}}(x , y\cdot k).
 	\end{align*}
 	Thus $\overline{\varphi_{\rho}}$ induces the linear map $\varphi_{\rho}: \S(M_1,\cN_1)_{\varphi^{*}(\rho)}\rightarrow \S(M_2,\cN_2)_{\rho}$.
 \end{proof}

 \subsection{Functoriality of the representation-reduced stated skein modules}

In this subsection, we will show the embedding between two marked three manifolds induces a linear map between the corresponding representation-reduced stated skein modules, and the  representation-reduced stated skein module of two disjoint marked three manifolds is isomorphic to the tensor product of their  representation-reduced stated skein modules.

 \begin{lemma}\label{embedding}
 	Let $f:(M_1,\cN_1)\rightarrow (M_2,\cN_2)$ be an embedding between two marked three manifolds. Then we have 
 	$f_*: \S(M_1,\cN_1)\rightarrow \S(M_2,\cN_2)$ behaves well with respect the Frobenius map.
 	 
 %	(a) We have $f_*: \sS_1(M_1,\cN_1)\rightarrow \sS_1(M_2,\cN_2)$ is an algebra homomorphism. Especially $f_*$ induces a map
 %	\begin{align*} f^{*}:\text{MaxSpec}(\sS_1(M_2,\cN_2))\rightarrow \text{MaxSpec}(\sS_1(M_1,\cN_1)),\;
 %		\rho\mapsto \rho\circ f_*,
 %		\end{align*}
 %	where $\rho \in \text{MaxSpec}(\sS_1(M_2,\cN_2))$ is regarded as an algebra homomorphism from $\sS_1(M_2,\cN_2)$ to $\mathbb{C}$.
 	
% 	(b) We have the following commutative diagram:
% 	 \begin{equation*}
% 		\begin{tikzcd}
% 			\sS_1(M_1,\cN_1)  \arrow[r, "f_*"]
% 			\arrow[d, "\cF"]  
% 			&  \sS_1 (M_2,\mathcal{N}_2) \arrow[d, "\cF"] \\
% 			\S(M_1,\cN_1)  \arrow[r, "f_*"] 
% 			&  \S(M_2,\cN_2)\\
% 		\end{tikzcd}.
% 	\end{equation*}
 	
% 	(c) For any $\alpha\in\sS_1(M_1,\cN_1),\beta\in\S(M_1,\cN_1)$, we have 
% 	$f_*(\alpha\cdot \beta) = f_*(\alpha)\cdot f_*(\beta)$.
% 	
% 	(d) For any $\rho \in \text{MaxSpec}(\sS_1(M_2,\cN_2))$, we have 
% 	$f_*: \S(M_1,\cN_1)\rightarrow \S(M_2,\cN_2)$ induces the linear map
% 	$f_{\rho}: \S(M_1,\cN_1)_{f^*(\rho)}\rightarrow \S(M_2,\cN_2)_{\rho}$.

 \end{lemma}
 \begin{proof}
 Condition (1)  is trivial.
 
 Condition (2): For any framed knot or stated framed arc $\alpha$ in $(M_1,\mathcal{N}_1)$ and any non-negative integer $k$, obviously, we have $f_*(\alpha^{(k)}) = (f_*(\alpha))^{(k)}$. Let $P(x) = \sum_{1\leq k\leq n} \lambda_k x^k \in\mathbb{C}[x]$, then we have 
 $$f_*(\alpha^{[P]}) =\sum_{1\leq k\leq n} \lambda_k f_*( \alpha^{(k)} )
  = \sum_{1\leq k\leq n} \lambda_k (f_*(\alpha))^{(k)} = f_*(\alpha)^{[P]}.$$
  This shows $f_*$ and $\cF$ commute with each other from the definition of $\cF$.
 
 Condition (3): For any disjoint stated $(M_1,\mathcal{N}_1)$-tangles
 $\alpha,\beta$, we have $$f_*(\alpha\cdot \beta) = f_*(\cF(\alpha)\cup\beta) = f_*(\cF(\alpha))\cup f_*(\beta) =\cF(f_*(\alpha))\cup f_*(\beta) = f_*(\alpha)\cdot f_*(\beta)$$
 (here we regard $\alpha$ as an element in $\sS_1(M_1,\mathcal{N}_1)$ and regard $\beta$ as an element in $\S(M_1,\mathcal{N}_1)$).
 %The proof is similar with the proof for Lemma 8.4 in \cite{wang2023stated}.
 
% (d) For any $\alpha\in \S(M_1,\cN_1)$ and $k\in\mathbb{C}$, define $f_{\rho}(\alpha\otimes k) = f_*(\alpha)\otimes k$. Then the well-definedness of $f_{\rho}$
% is implied by (c).

 \end{proof}
 
% We know the algebra homomorphism $$f_*:\sS_1(M_1,\cN_1)\rightarrow \sS_1(M_2,\cN_2)$$ induces a map $f^{*} :
% \text{MaxSpec}(\sS_1(M_2,\cN_2))\rightarrow \text{MaxSpec}(\sS_1(M_1,\cN_1))$. 
For any $\rho\in \text{MaxSpec}(\sS_1(M_2,\cN_2))$, Lemma \ref{con} implies that
 $f_*: \S(M_1,\cN_1)\rightarrow \S(M_2,\cN_2)$ induces a linear map
 $f_{\rho} : \S(M_1,\cN_1)_{f^{*}(\rho)}\rightarrow \S(M_2,\cN_2)_{\rho}$.
 
 Let $(M_1,\cN_1)$ and $(M_2,\cN_2)$ be two marked three manifolds. Then for each $q^{\frac{1}{2}}$, there exists a linear isomorphism 
 $$\kappa:\S(M_1,\cN_1)\otimes \S(M_2,\cN_2)\rightarrow \S((M_1,\cN_1)\cup(M_2,\cN_2)),$$ defined by 
 $\kappa(\alpha\otimes \beta) = \alpha \cup \beta$, where 
 $\alpha$ is a stated $(M_1,\mathcal{N}_1)$-tangle and $\beta$ is a stated
 $(M_2,\mathcal{N}_2)$-tangle.
 \begin{lemma}
 	We have $$\kappa:\S(M_1,\cN_1)\otimes \S(M_2,\cN_2)\rightarrow \S((M_1,\cN_1)\cup(M_2,\cN_2))$$ behaves well with respect to the Frobenius map (the Frobenius map for $\S(M_1,\cN_1)\otimes \S(M_2,\cN_2)$ is defined by sending $x\otimes y$ to $\cF(x)\otimes \cF(y)$, where $x\in \sS_1(M_1,\cN_1)$ and $y\in \sS_1(M_2,\cN_2)$). 
 \end{lemma}
 \begin{proof}
 	The proof is trivial.
 \end{proof}
 
 %Since $\kappa:\sS_1(M_1,\cN_1)\otimes \sS_1(M_2,\cN_2)\rightarrow \sS_1((M_1,\cN_1)\cup(M_2,\cN_2))$ is an algebra isomorphism,  it induces a bijection $$\kappa^{*}:\text{MaxSpec}(\sS_1((M_1,\cN_1)\cup(M_2,\cN_2)))\rightarrow \text{MaxSpec}(\sS_1(M_1,\cN_1)\otimes \sS_1(M_2,\cN_2)).$$

It is easy to show that $\text{MaxSpec}(\sS_1(M_1,\cN_1)\otimes \sS_1(M_2,\cN_2)) = \text{MaxSpec}(\sS_1(M_1,\cN_1))\times \text{MaxSpec}( \sS_1(M_2,\cN_2))$, that is, any $\rho\in \text{MaxSpec}(\sS_1(M_1,\cN_1)\otimes \sS_1(M_2,\cN_2))$ is a pair $(\rho_1,\rho_2)$, where $\rho_i\in \text{MaxSpec}(\sS_1(M_i,\cN_i))$ for 
 $i=1,2$ (here we regard $(\rho_1,\rho_2)$ as an algebra homomorphism from $\sS_1(M_1,\cN_1)\otimes \sS_1(M_2,\cN_2)$ to $\mathbb{C}$, define by $(\rho_1,\rho_2)(x_1\otimes x_2) = \rho_1(x_1)\rho_2(x_2)$, where $x_i\in \sS_1(M_i,\cN_i)$ for $i=1,2$).
 
 For any $\rho\in \text{MaxSpec}(\sS_1((M_1,\cN_1)\cup(M_2,\cN_2)))$ and $i=1,2$, we use $\rho_i$ to denote the composition 
 $\sS_1(M_i,\mathcal{N}_i)\rightarrow \sS_1((M_1,\cN_1)\cup(M_2,\cN_2))\rightarrow \mathbb{C}$, where the first map is induced by the embedding from $(M_i,\mathcal{N}_i)$ to $(M_1,\cN_1)\cup(M_2,\cN_2)$ and the second map is $\rho$.
 Then $\kappa^*(\rho) = (\rho_1,\rho_2)$.
 
  For any $\rho\in \text{MaxSpec}(\sS_1((M_1,\cN_1)\cup(M_2,\cN_2)))$ with $\kappa^*(\rho) = (\rho_1,\rho_2)$ and any $q^{\frac{1}{2}}$, the linear isomorphism, $$\kappa:\S(M_1,\cN_1)\otimes \S(M_2,\cN_2)\rightarrow \S((M_1,\cN_1)\cup(M_2,\cN_2))$$ induces the following linear isomorphism
  \begin{equation}\label{union}
 \kappa_{\rho}:\S(M_1,\cN_1)_{\rho_1}\otimes \S(M_2,\cN_2)_{\rho_2}\rightarrow \S((M_1,\cN_1)\cup(M_2,\cN_2))_{\rho},
 \end{equation}
 defined by $\kappa_{\rho}((x_1\otimes k_1)\otimes (x_2\otimes k_2))
  = \kappa(x_1\otimes x_2)\otimes k_1k_2,$ where $k_1,k_2\in\mathbb{C}$ and 
  $x_i\in \S(M_i,\cN_i)$ for $i=1,2$.
% where $$
 
 \section{The splitting map for representation-reduced stated skein modules}\label{666}
The existence of the splitting map  is a very important property for stated skein modules. Thang L{\^e} constructed the quantum trace map using the splitting map \cite{le2018triangular} (the quantum trace map was originally constructed in \cite{bonahon2011quantum}). The splitting map also gives the stated skein algebra of the bigon a comultiplication structure (it actually has a Hopf algebra structure) \cite{costantino2022stated1}. 

In this section, we prove the splitting map for stated skein modules induces the splitting map for the representation-reduced stated skein modules. We also show the splitting map for the representation-reduced stated skein modules is injective if there exists at  least one component of $\mathcal{N}$ such that this component and the boundary of the splitting disk belong  to the same component of $\partial M$. %when every component of the marked three manifold contains at least one marking.

\subsection{The splitting map}\label{a}
Let $\MN$ be a marked three manifold, and let $(D,u)$ be a pair, where 
$D$, called the splitting disk, is a properly embedded disk in $M$ and $u$ is an embedded oriented open interval in $D$. Suppose $U(D)$ is an open regular neighborhood of $D$ such that $U(D)$ is isomorphic to $D\times (0,1)$ and $\partial U(D) = \partial D\times (0,1)$.  Let $M^{'} = M\setminus U(D)$. Then there exists a projection $\text{pr}:M^{'}\rightarrow M$. Suppose $\text{pr}^{-1}(u) =u_1\cup u_2$, where both $u_1$ and $u_2$ are oriented open intervals in $\partial M^{'}$. Define $\text{Cut}_{(D,u)}\MN = (M^{'},\mathcal{N}^{'})$, where $\mathcal N^{'} = \mathcal N\cup u_1\cup u_2$.

For any stated $\MN$-tangle $\alpha$, we isotope $\alpha$ such that $\alpha\cap D =\alpha\cap u$ and at each point in $\alpha\cap u$ the framing of $\alpha$ is given by the the velocity vector of $u$. Let $s$ be a map from $\alpha\cap u$ to $\{-,+\}$.
We define a stated $(M^{'},\mathcal{N}^{'})$-tangle $\alpha(s)$ in the following way:
 the  $(M^{'},\mathcal{N}^{'})$-tangle is $\text{pr}^{-1}(\alpha)$; the states  for $\text{pr}^{-1}(\alpha)\cap \mathcal{N}$ are inherited from $\alpha$; for each point  $u\in\alpha\cap u$, we state the two endpoints $\text{pr}^{-1}(u)$ with $s(u)$. Then there exists a linear map
 $\Theta_{(D,u)}:\S\MN\rightarrow \S(M^{'},\mathcal{N}^{'})$ such that 
 $\Theta_{(D,u)}(\alpha)=\sum_{s:\alpha\cap u\rightarrow \{-,+\}} \alpha(s)$ \cite{bloomquist2020chebyshev}.  When there is no confusion, we can omit the subscript for $\Theta_{(D,u)}$.

 \begin{lemma}\label{splitting}
 	We have $\Theta:\S\MN\rightarrow \S(M^{'},\mathcal{N}^{'})$ behaves well with respect to the Frobenius map.
 \end{lemma}
 \begin{proof}
 	Condition (1) is trivial from the definition of $\Theta$.
 	
 	Condition (2) is proved in \cite{bloomquist2020chebyshev}.
 	
 	Condition (3): For any disjoint stated $\MN$-tangles $\alpha$ and $\beta$, we have  $$\Theta(\alpha\cdot\beta)=\Theta(\cF(\alpha)\cup\beta) = \Theta(\cF(\alpha))\cup\Theta(\beta) = \cF(\Theta(\alpha))\cup \Theta(\beta) = \Theta(\alpha)\cdot \Theta(\beta),$$
 	where we regard $\alpha$ as an element in $\mno$ and regard $\beta$ as an element in $\S\MN$.
 	
 \end{proof}

% We also have the following commutative diagram \cite{bloomquist2020chebyshev}:
% \begin{equation}\label{diag}
% \begin{tikzcd}
% 	\mno  \arrow[r, "\Theta"]
% 	\arrow[d, "\cF"]  
% 	&  \sS_1 (M^{'},\mathcal{N}^{'}) \arrow[d, "\cF"] \\
% 	\S\MN  \arrow[r, "\Theta"] 
% 	&  \S(M^{'},\mathcal{N}^{'})\\
% \end{tikzcd}.
% \end{equation}
% 
% The commutative diagram in equation \eqref{diag} implies the following Lemma.
% 
% \begin{lemma}
% 	For any $\alpha\in\mno,\beta\in\S\MN$, we have $\Theta(\alpha\cdot\beta)
% 	= \Theta(\alpha\cdot\beta)$.
%  \end{lemma}
%  \begin{proof}
%  	We can suppose $\alpha$ and $\beta$ are disjoint stated $\MN$-tangles. Then $$\Theta(\alpha\cdot\beta)=\Theta(\cF(\alpha)\cup\beta) = \Theta(\cF(\alpha))\cup\Theta(\beta) = \cF(\Theta(\alpha))\cup \Theta(\beta) = \Theta(\alpha)\cdot \Theta(\beta).$$
%  \end{proof}
%  

 %Obviously $\Theta:\mno\rightarrow \sS_1(M^{'},\mathcal{N}^{'})$ induces a map
 %$$\Theta^{*}:\text{MaxSpec}(\sS_1(M^{'},\mathcal{N}^{'}))\rightarrow\max,$$
 %defined by $\Theta^{*}(\rho) = \rho\circ\Theta$ for any $\rho\in \text{MaxSpec}(\sS_1(M^{'},\mathcal{N}^{'}))$ (here we regard $\rho$ as an algebra homomorphism from $\sS_1(M^{'},\mathcal{N}^{'})$ to $\mathbb{C}$).

\begin{remark}
 There is a surjective algebra homomorphism $\Phi:\mno\rightarrow R_2\MN,$ where $R_2\MN$ is the coordinate ring of some algebraic set, Theorem 3.15 in \cite{wang2023stated}. We  have $\text{Ker}\,\Phi$ consists of all nilpotents of $\mno$, Theorem 5.21 in \cite{wang2023stated}. Then $\Phi$ induces a bijection $\Phi^{*}: \text{MaxSpec}(R_2\MN)\rightarrow \text{MaxSpec}(\mno)$.
Proposition 3.18 in \cite{wang2023stated} implies there is a surjective map $$\nu^{*}:\text{MaxSpec}(R_2(M^{'},\mathcal{N}^{'}))\rightarrow \text{MaxSpec}(R_2\MN)$$ induced by an algebra homomorphism $\nu:R_2\MN\rightarrow R_2(M^{'},\mathcal{N}^{'})$. Theorem 3.19 in \cite{wang2023stated} shows $\Phi\circ\Theta = \nu\circ \Phi$.
Thus $\Theta^{*}\circ\Phi^{*} = \Phi^{*}\circ \nu^{*}$. Then we have $\Theta^{*}$ is surjective since $\Phi^{*}$ is a bijection and $\nu^{*}$ is surjective. 
\end{remark}

 For any $\q$ and any $\rho\in \text{MaxSpec}(\sS_1(M^{'},\mathcal{N}^{'}))$, Lemma \ref{con} implies that $\Theta:\S\MN\rightarrow \sS_{q^{1/2}}(M^{'},\mathcal{N}^{'})$ induces a linear map
 $$\Theta_{\rho}:\S\MN_{\Theta^{*}(\rho)}\rightarrow \S(M^{'},\mathcal{N}^{'})_{\rho}.$$
 We will call $\Theta_{\rho}$ the splitting map for the representation-reduced stated skein module.
 Suppose $\partial D$, where $D$ is the splitting disk, is contained in the boundary component $V$.  In subsection \ref{inj}, we will prove $\Theta_{\rho}$ is injective when $V\cap\cN\neq \emptyset$.

\subsection{Gluing the thickening of the marked triangle}
\def \bT {\mathbb{T}}

Let $\bT$ denote the marked three manifold in the following picture:
$$
\raisebox{-.35in}{
	\begin{tikzpicture}%[dline /. style ={line width =2pt}]
		\tikzset{->-/.style=
			{decoration={markings,mark=at position #1 with
					{\arrow{latex}}},postaction={decorate}}}
		\draw[line width = 1pt] (0,0) rectangle (4.8, 2);
		\draw [dashed] (0,0)--(2.4,2.4);
		\draw [dashed] (2.4,2.4)--(2.4,4.4);
		\draw [dashed] (2.4,2.4)--(4.8,0);
		\draw [line width = 1pt] (0,2)--(2.4,4.4);
		\draw [line width = 1pt] (2.4,4.4)--(4.8,2);
		\draw [color=red, line width = 1pt] (1.2,1.2)--(1.2,1.9);
		\draw [color=red, line width = 1pt] (1.2,2.1)--(1.2,3.2);
		\draw [color=red,->, line width = 1pt] (1.2,2.1)--(1.2,2.8);
		\draw [color=red, line width = 1pt] (3.6,1.2)--(3.6,1.9);
		\draw [color=red, line width = 1pt] (3.6,2.1)--(3.6,3.2);
		\draw [color=red,->, line width = 1pt] (3.6,2.1)--(3.6,2.8);
		\draw [color=red, line width = 1pt] (2.6,0)--(2.6,2);
		\draw [color=red,->, line width = 1pt] (2.6,0)--(2.6,1.6);
		%\node at(1,2) {$n$};
		\node[right] at(1.2,2.5){$e_1$};
		\node[left] at(3.6,2.5){$e_2$};
		\node[right] at(2.6,1.3){$e_3$};
\end{tikzpicture}}
$$

Then $\bT$ is the thickening of $\mathfrak T$.

%Recall that $\fT$ denotes
%the standard ideal triangle. Then $\Delta$, after removing the three vertical edges, is isomorphic to $\fT\times [-1,1]$. 

For each $i=1,2,3,$ let $D_{i}$ be an embedded disk in $\partial \bT$ such that the closure of $e_i$ is contained in the interior of $D_i$ and there is no intersection among these three disks. From now on, when we draw $\bT$, we may omit all the black lines, that is, we only draw three red arrows. We also only draw involved markings and stated tangles when we try to draw stated tangles in marked three manifolds.

Let $\MN$ be any marked three manifold with $\sharp \mathcal{N}\geq 2$. Suppose $e_1^{'},e_2^{'}$ are two components of $\mathcal{N}$. For each $i=1,2$, let $D_{i}^{'}$ be an embedded disk on the boundary of $M$ such that the intersection between the closure of $\mathcal{N}$ and $D_i$ is the closure of $e_i$ and the closure of $e_i$ is contained in the interior of $D_i$ and $D_1\cap D_2 =\emptyset$. For each $i=1,2$, let $\phi_{i} : D_i^{'} \rightarrow D_i$ be a diffeomorphism such that $\phi_i(e_{i}^{'}) = e_i$ and $\phi_i$ preserves the orientations of $e_i$ and $e_i^{'}$. We set 
$$M_{e_1^{'}\bT e_2^{'}} = (M\cup \bT)/(\phi_i(x) = x, x\in D_i^{'},i=1,2),\;
\mathcal{N}_{e_1^{'}\bT e_2^{'}} = (\mathcal{N}-(e_1^{'}\cup e_2^{'}))\cup e_3.$$
Then $(M_{e_1^{'}\bT e_2^{'}},\mathcal{N}_{e_1^{'}\bT e_2^{'}})$ is a marked three manifold. We use 
$(M,\mathcal{N})_{e_1^{'}\bT e_2^{'}}$ to denote  this marked three manifold.

Then there is a linear isomorphism 
$QF_{e_1^{'},e_2^{'}} : \S\MN\rightarrow \S(M_{e_1^{'}\bT e_2^{'}},\mathcal{N}_{e_1^{'}\bT e_2^{'}})$ \cite{costantino2022stated1,wang2023stated}.
%L{\^e} and Sikora  introduced this map when $(M,N)$ is the thickening of a punctured bordered surface
%\cite{le2021stated}. The definition here for $QF_{e_1^{'},e_2^{'}}$ is similar with the one defined in \cite{le2021stated} for punctured bordered surface.
 We use $\iota$ to denote the obvious embedding from $M$ to 
$M_{e_1^{'}\bT e_2^{'}}$.  For any stated $\MN$-tangles, we extend the ends of $\iota(\alpha)$ on each $e_i,i=1,2$, to $e_3$ such that the framing of extended parts contained in $\bT$
is given by the positive direction of $[0,1]$ and all the ends on $e_3$ extended from $e_1$  are higher than all the ends extended from $e_2$. To be precise, see the following picture:
$$
\raisebox{-.60in}{
	\begin{tikzpicture}%[dline /. style ={line width =2pt}]
		\tikzset{->-/.style=
			{decoration={markings,mark=at position #1 with
					{\arrow{latex}}},postaction={decorate}}}
		\draw [color=red, line width = 1pt] (0,0)--(0,2);
		%\draw [color=red, line width = 1pt] (1.2,2.1)--(1.2,3.2);
		\draw [color=red,->, line width = 1pt] (0,0)--(0,1);
		\draw [color=red, line width = 1pt] (1.2,-2)--(1.2,0);
		%\draw [color=red, line width = 1pt] (3.6,2.1)--(3.6,3.2);
		\draw [color=red,->, line width = 1pt] (1.2,-2)--(1.2,-1);
		\draw [color=red, line width = 1pt] (2.4,0)--(2.4,2);
		\draw [color=red,->, line width = 1pt] (2.4,0)--(2.4,1);
		\draw [color=blue, line width = 1pt] (-1,1.2)--(0,1.2);
		\draw [color=blue, line width = 1pt] (-1,1.8)--(0,1.8);
		\draw [color=blue, line width = 1pt] (2.4,0.2)--(3.4,0.2);
		\draw [color=blue, line width = 1pt] (2.4,0.8)--(3.4,0.8);
		\node[left] at(0,1.6){$\vdots$};
		\node[right] at(2.4,0.6){$\vdots$};
		\node[right] at(0,1.2){\small $i_1$};
		\node[right] at(0,1.8){\small $i_{k_1}$};
		\node[left] at(2.4,0.2){\small $j_1$};
		\node[left] at(2.4,0.8){\small $j_{k_2}$};
\end{tikzpicture}}
\longrightarrow
\raisebox{-.60in}{
	\begin{tikzpicture}%[dline /. style ={line width =2pt}]
		\tikzset{->-/.style=
			{decoration={markings,mark=at position #1 with
					{\arrow{latex}}},postaction={decorate}}}
		\draw [color=red, line width = 1pt] (0,0)--(0,2);
		%\draw [color=red, line width = 1pt] (1.2,2.1)--(1.2,3.2);
		\draw [color=red,->, line width = 1pt] (0,0)--(0,1);
		\draw [color=red, line width = 1pt] (1.2,-2)--(1.2,0);
		%\draw [color=red, line width = 1pt] (3.6,2.1)--(3.6,3.2);
		\draw [color=red,->, line width = 1pt] (1.2,-2)--(1.2,-1);
		\draw [color=red, line width = 1pt] (2.4,0)--(2.4,2);
		\draw [color=red,->, line width = 1pt] (2.4,0)--(2.4,1);
		\draw [color=blue, line width = 1pt] (-1,1.2)--(0,1.2);
		\draw [color=blue, line width = 1pt] (-1,1.8)--(0,1.8);
		\draw [color=blue, line width = 1pt] (2.4,0.2)--(3.4,0.2);
		\draw [color=blue, line width = 1pt] (2.4,0.8)--(3.4,0.8);
		\draw [color=blue, line width = 1pt] (1.2,-0.8)--(0,1.2);
		\draw [color=blue, line width = 1pt] (1.2,-0.2)--(0,1.8);
		\draw [color=blue, line width = 1pt] (1.2, -1.8)--(2.4,0.2);
		\draw [color=blue, line width = 1pt] (1.2,-1.2)--(2.4,0.8);
		\node[left] at(0,1.6){$\vdots$};
		\node[right] at(2.4,0.6){$\vdots$};
		\node[right] at(1.2,-0.8){\small $i_1$};
		\node[right] at(1.2,-0.2){\small $i_{k_1}$};
		\node[left] at(1.2, -1.8){\small $j_1$};
		\node[left] at(1.2,-1.2){\small $j_{k_2}$};
\end{tikzpicture}}
$$
where the blue lines are parts of stated tangles  and the framing in the picture is given by the red arrows.
 To simplify notation, we normally omit the subscript for $QF_{e_1^{'},e_2^{'}}$ when there is no confusion.

\begin{lemma}\label{QF}
	We have $QF : \S\MN\rightarrow \S(M_{e_1^{'}\bT e_2^{'}},\mathcal{N}_{e_1^{'}\bT e_2^{'}})$ behaves well with respect the Frobenius map.
%	
%	Let $\MN$ be any marked three manifold with $\sharp \mathcal{N}\geq 2$. Suppose $e_1^{'},e_2^{'}$ are two components of $\mathcal{N}$. Then we have the following results:
%	
%	(a) We have $QF: \sS_1(M,\cN)\rightarrow \sS_1(M_{e_1^{'}\bT e_2^{'}},\mathcal{N}_{e_1^{'}\bT e_2^{'}})$ is an algebra isomorphism. Especially it induces a bijection
%	\begin{align*} QF^{*}:\text{MaxSpec}(\sS_1(M_{e_1^{'}\bT e_2^{'}},\mathcal{N}_{e_1^{'}\bT e_2^{'}}))\rightarrow \text{MaxSpec}(\sS_1(M,\cN)),\;
%		\rho\mapsto \rho\circ QF,
%	\end{align*}
%	where $\rho \in \text{MaxSpec}(\sS_1(M_{e_1^{'}\bT e_2^{'}},\mathcal{N}_{e_1^{'}\bT e_2^{'}}))$ is regarded as an algebra homomorphism from $\sS_1(M_{e_1^{'}\bT e_2^{'}},\mathcal{N}_{e_1^{'}\bT e_2^{'}})$ to $\mathbb{C}$.
%	
%	(b) We have the following commutative diagram:
%	\begin{equation*}
%		\begin{tikzcd}
%			 \sS_1(M,\cN)  \arrow[r, "QF"]
%			\arrow[d, "\cF"]  
%			&  \sS_1(M_{e_1^{'}\bT e_2^{'}},\mathcal{N}_{e_1^{'}\bT e_2^{'}}) \arrow[d, "\cF"] \\
%			\S(M,\cN)  \arrow[r, "QF"] 
%			&  \S(M_{e_1^{'}\bT e_2^{'}},\mathcal{N}_{e_1^{'}\bT e_2^{'}})\\
%		\end{tikzcd}.
%	\end{equation*}
%	
%	(c) For any $\alpha\in\sS_1(M,\cN),\beta\in\S(M,\cN)$, we have 
%	$QF(\alpha\cdot \beta) = QF(\alpha)\cdot QF(\beta)$.
%	
%	(d) For any $\rho \in \text{MaxSpec}(\sS_1(M_2,\cN_2))$, we have 
%	$QF: \S(M_1,\cN_1)\rightarrow \S(M_2,\cN_2)$ induces the linear isomorphism
%	$QF_{\rho}: \S(M,\cN)_{QF^*(\rho)}\rightarrow \S(M_{e_1^{'}\bT e_2^{'}},\mathcal{N}_{e_1^{'}\bT e_2^{'}})_{\rho}$.
	
\end{lemma}\label{666666}
\begin{proof}
	Condition (1) is trivial from the defintion of the map $QF$.
	
	Condition (2): We know $\cF(\alpha) = \alpha^{(N)}$ for any stated arc $\alpha$. Then the operation of taking $N$ parallel copies and the operate of $QF$ commute with each other. This completes the proof for condition (2).
	
	Condition (3): The proof is similar with Lemma \ref{splitting}.
	
	%(d) The proof is similar with the proof for (d) in Lemma \ref{embedding}.
\end{proof}

%Lemma \ref{666666} implies  $QF: \sS_1(M,\cN)\rightarrow \sS_1(M_{e_1^{'}\bT e_2^{'}},\mathcal{N}_{e_1^{'}\bT e_2^{'}})$  induces a bijection
%	\begin{align*} QF^{*}:\text{MaxSpec}(\sS_1(M_{e_1^{'}\bT e_2^{'}},\mathcal{N}_{e_1^{'}\bT e_2^{'}}))\rightarrow \text{MaxSpec}(\sS_1(M,\cN)),\;
%			\rho\mapsto \rho\circ QF,
%		\end{align*}
%	where $\rho \in \text{MaxSpec}(\sS_1(M_{e_1^{'}\bT e_2^{'}},\mathcal{N}_{e_1^{'}\bT e_2^{'}}))$ is regarded as an algebra homomorphism from $\sS_1(M_{e_1^{'}\bT e_2^{'}},\mathcal{N}_{e_1^{'}\bT e_2^{'}})$ to $\mathbb{C}$.

For any $\q$ and $\rho\in \text{MaxSpec}(\sS_1(M_{e_1^{'}\bT e_2^{'}},\mathcal{N}_{e_1^{'}\bT e_2^{'}}))$, Lemma \ref{con} implies $QF : \S\MN\rightarrow \S(M_{e_1^{'}\bT e_2^{'}},\mathcal{N}_{e_1^{'}\bT e_2^{'}})$ induces the linear isomorphism
$$QF_{\rho} : \S\MN_{QF^{*}(\rho)}\rightarrow \S(M_{e_1^{'}\bT e_2^{'}},\mathcal{N}_{e_1^{'}\bT e_2^{'}})_{\rho}.$$

\subsection{Adding an extra marking}
Let $\MN$ be a marked three manifold, and let $e$ be an embedded oriented open interval in $\partial M$ such that there is no intersection between the closure of $e$ and the closure of $\mathcal{N}$. Define a new marked three manifold $(M,\mathcal{N}^{'})$, where $\mathcal{N}^{'} = \mathcal{N}\cup e$. We say $(M,\mathcal{N}^{'})$ is obtained from $\MN$ by adding one extra marking $e$. We use $l:\MN\rightarrow (M,\mathcal{N}^{'})$ to denote the obvious embedding. %We can omit the subscript for $l_e$ when there is no confusion.

\def \bB {\mathbb{B}}

Let $\bB$ be the three dimensional solid ball with two markings on its boundary, that is, $\bB$ is the thickening of the bigon. We label one marking of $\bB$ as $b$. %It is well-known that
%$S_n(B,v)= O_q(SL_2)$.

\begin{lemma}\label{key2}
	Let $\MN$ be a marked three manifold, and let $(M,\mathcal{N}^{'})$ be obtained from $\MN$ by adding one extra marking $e$. Suppose $e$ is contained in the component $U$ of $\partial M$ and $U\cap\mathcal{N} \neq\emptyset$. Then we have the following results:
	
	(a) For any $\rho\in\text{MaxSpec}(\sS_1(M,\mathcal{N}^{'}))$, we have
	$$l_{\rho}:\S\MN_{l^{*}(\rho)}\rightarrow \S(M,\mathcal{N}^{'})_{\rho}$$ is injective.
	
	(b) For any $\rho\in\text{MaxSpec}(\sS_1(M,\mathcal{N}^{'}))$,  there exists $\rho^{'}\in \text{MaxSpec}(\sS_1(\mathbb{B}))$ such that
	$$\S(M,\mathcal{N}^{'})_{\rho}\simeq\S\MN_{l^{*}(\rho)}\otimes \S(\mathbb{B})_{\rho^{'}}.$$  Especially, we have $$\dim_{\bC} \S(M,\mathcal{N}^{'})_{\rho} = N^3 \dim_{\bC} \S\MN_{l^{*}(\rho)}.$$
	
\end{lemma}
\begin{proof}
	This Lemma is implied by Theorem 6.10 in \cite{wang2023stated} since all the maps in Theorem 6.10 in \cite{wang2023stated} behave well with respect to the Frobenius map as shown in Lemmas \ref{embedding}, \ref{QF}.

From Theorem 4.10 in \cite{wang2023frobenius}, we know $\dim_{\bC} \S(\mathbb{B})_{\rho^{'}}= N^3$. So we have 
$$\dim_{\bC} \S(M,\mathcal{N}^{'})_{\rho} = N^3 \dim_{\bC} \S\MN_{l^{*}(\rho)}.$$
\end{proof}
\def \bT {\mathbb{T}}
%We will use the techniques used in Theorem 6.10 in \cite{wang2023stated} to prove this Lemma.
% Since $U\cap \mathcal{N}\neq \emptyset$, we can choose one component $a$ of $\mathcal{N}$ such that $a\subset V$. Then $((M,\mathcal{N})\cup \bB)_{a\mathbb{T} b}\simeq (M,\cN^{''})$ where $(M,\cN)\cup \bB$ is taking the disjoint union and $\cN^{''}$ is obtained from $\cN$ by adding one extra marking in $U$. We use $F$ to denote the isomorphism  from $(M,\cN^{'})$ to $((M,\cN)\cup \bB)_{a\mathbb{T} b}$, use 
%$\iota$ to denote the embedding from $(M,\cN)$ to $(M,\cN)\cup \bB$. It is easy to show we have the following commutative diagram:
%$$\begin{tikzcd}
	%	\S\MN  \arrow[r, "\Theta_{*}"]
	%	\arrow[d, "l_{*}"]  
	%	&  \S((M,N)\cup \bB,v)  \arrow[d, "QF"] \\
	%	\S(M,\cN^{'})  \arrow[r, "F_{*}"] 
	%	&  \S(((M,\cN)\cup \bB)_{a\bT b},v)\\
	%\end{tikzcd}.$$
	%Then $l_{ad}$ is injective since both $QF$ and $J_{*}$ are isomorphisms and $L_{*}$ is injective.
	%Thus $$\text{Ker\,}((M,N),V,v) = \text{Ker\,}l_{ad} =0.$$
	
	%(b) We have $((M,N)\cup B)_{a\Delta b}\simeq (M,N^{''})\simeq (M,N^{'})$, thus
	%$$S_n(M,N^{'},v)\simeq S_n(((M,N)\cup B)_{a\Delta b},v)\simeq S_n(M,N,v)\otimes O_q(SLn).$$

\subsection{Injectivity for the splitting map}\label{inj}
In this subsection we will  show the splitting map for the representation-reduced stated skein modules is injective when the boundary component of the three manifold, containing the boundary of the splitting disk, contains at least one marking.

We define $c_+(\q) = q^{\frac{1}{2}}$ and $c_{-}(\q) = -(q^{\frac{1}{2}})^5$. We define a map $\overline{\cdot}:\{-,+\}\rightarrow \{-,+\}$ such that 
$\overline{+} = -$ and $\overline{-} = +$. Note that $c_+(\q)^{N} = c_+(1)$ and $c_-(\q)^N = c_-(1)$.

We use $\bar H$ to denote the negative half-twist for parallel strands, please see Figure \ref{twist}.

\begin{figure}[h]
	\centering
	\includegraphics[scale=0.5]{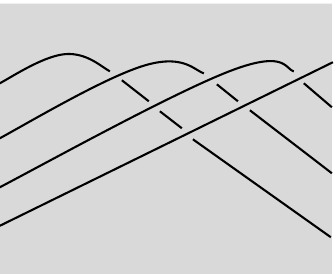}
	\caption{The negative half-twist for four strands.}\label{twist}
\end{figure}

\begin{lemma}(\cite{le2021stated})
	Let $\MN$ be a marked three manifold, and $e$ be a component of $\cN$. Then there is a linear isomorphism
	$h_{e}: \S\MN\rightarrow \S\MN$  given by
	$$
	h_{e}\left(
	\raisebox{-.40in}{
		
		\begin{tikzpicture}%[dline /. style ={line width =2pt}]
			\tikzset{->-/.style=
				
				{decoration={markings,mark=at position #1 with
						
						{\arrow{latex}}},postaction={decorate}}}

			\filldraw[draw=white,fill=gray!20] (-1.5,0) rectangle (0, 2.5);
			\draw [line width =1.5pt,decoration={markings, mark=at position 1 with {\arrow{>}}},postaction={decorate}](0,2.5)--(0,0);
			\draw[line width =1pt] (-1.5,0.5)--(0,0.5);
			\draw[line width =1pt] (-1.5,1)--(0,1);
			\draw[line width =1pt] (-1.5,2)--(0,2);
			\node [left] at(0,1.5) {$\vdots$};
			\node [right] at(0,0.5) {$i_1$};
			\node [right] at(0,1) {$i_2$};
			\node [right] at(0,2) {$i_k$};
		\end{tikzpicture}
	}\right)= \left(\frac 1{ \prod_{j=1}^k c_{{i_j}}(\q)}\right) \cdot
	\raisebox{-.40in}{
		
		\begin{tikzpicture}%[dline /. style ={line width =2pt}]
			\tikzset{->-/.style=
				
				{decoration={markings,mark=at position #1 with
						
						{\arrow{latex}}},postaction={decorate}}}

			\filldraw[draw=white,fill=gray!20] (-1.5,0) rectangle (0, 2.5);
			\draw [line width =1pt] (-1,0.2) rectangle (-0.5, 2.3);
			\draw [line width =1.5pt,decoration={markings, mark=at position 1 with {\arrow{>}}},postaction={decorate}](0,2.5)--(0,0);
			\draw[line width =1pt] (-1.5,0.5)--(-1,0.5);
			\draw[line width =1pt] (-0.5,0.5)--(0,0.5);
			\draw[line width =1pt] (-1.5,1)--(-1,1);
			\draw[line width =1pt] (-0.5,1.5)--(0,1.5);
			\draw[line width =1pt] (-1.5,2)--(-1,2);
			\draw[line width =1pt] (-0.5,2)--(0,2);
			\node [left] at(0,1) {$\vdots$};
			\node [left] at(-1,1.5) {$\vdots$};
			\node [right] at(0,0.5) {$\overline{i_k}$};
			\node [right] at(0,1.5) {$\overline{i_2}$};
			\node [right] at(0,2) {$\overline{i_1}$};
			\node  at(-0.75,1.25) {$\bar{H}$};
		\end{tikzpicture}
	}
	$$
	where $i_j\in\{-,+\}$, for $1\leq j\leq k$, and  the  thick line with an arrow is a part of the marking $e$.
\end{lemma}  

%The proof for the following Lemma is trivial.

\begin{lemma}\label{change}
	Let $\MN$ be a marked three manifold, and $e$ be a component of $\cN$. We have
	$h_{e}: \S\MN\rightarrow \S\MN$ behaves well with respect the Frobenius map.
\end{lemma}
\begin{proof}
	Condition (1) is trivial.
	
	Condition (2): 
	We use 
	$\raisebox{-.15in}{
		
		\begin{tikzpicture}%[dline /. style ={line width =2pt}]
			\tikzset{->-/.style=
				
				{decoration={markings,mark=at position #1 with
						
						{\arrow{latex}}},postaction={decorate}}}

			\filldraw[draw=white,fill=gray!20] (-1.5,0.1) rectangle (0, 0.9);
			\draw [line width =1.5pt,decoration={markings, mark=at position 1 with {\arrow{>}}},postaction={decorate}](0,0.9)--(0,0.1);
			\draw[line width =1pt] (-1.5,0.5)--(0,0.5);
			\draw [fill=gray!20] (-1,0.5) circle (0.2);
			\node [right] at(0,0.5) {$i$};
				\node  at(-1,0.5) {\small $N$};
		\end{tikzpicture}
	}$ 
	to denote $N$ parallel copies of the corresponding stated arc (they have the same state). Thus $\cF(\raisebox{-.15in}{
		
		\begin{tikzpicture}%[dline /. style ={line width =2pt}]
			\tikzset{->-/.style=
				
				{decoration={markings,mark=at position #1 with
						
						{\arrow{latex}}},postaction={decorate}}}

			\filldraw[draw=white,fill=gray!20] (-1.5,0.1) rectangle (0, 0.9);
			\draw [line width =1.5pt,decoration={markings, mark=at position 1 with {\arrow{>}}},postaction={decorate}](0,0.9)--(0,0.1);
			\draw[line width =1pt] (-1.5,0.5)--(0,0.5);
%			\draw [fill=gray!20] (-1,0.5) circle (0.2);
			\node [right] at(0,0.5) {$i$};
%			\node  at(-1,0.5) {\small $N$};
		\end{tikzpicture}
	}) =\raisebox{-.15in}{
	
	\begin{tikzpicture}%[dline /. style ={line width =2pt}]
		\tikzset{->-/.style=
			
			{decoration={markings,mark=at position #1 with
					
					{\arrow{latex}}},postaction={decorate}}}

		\filldraw[draw=white,fill=gray!20] (-1.5,0.1) rectangle (0, 0.9);
		\draw [line width =1.5pt,decoration={markings, mark=at position 1 with {\arrow{>}}},postaction={decorate}](0,0.9)--(0,0.1);
		\draw[line width =1pt] (-1.5,0.5)--(0,0.5);
		\draw [fill=gray!20] (-1,0.5) circle (0.2);
		\node [right] at(0,0.5) {$i$};
		\node  at(-1,0.5) {\small $N$};
	\end{tikzpicture}
	}$.
We use $\raisebox{-.40in}{\begin{tikzpicture}%[dline /. style ={line width =2pt}]
		\tikzset{->-/.style=				
			{decoration={markings,mark=at position #1 with
					{\arrow{latex}}},postaction={decorate}}}
		\filldraw[draw=white,fill=gray!20] (-1.5,0) rectangle (1, 2.5);
		\draw [line width =1pt] (-1,0.2) rectangle (-0.5, 2.3);
%		\draw [line width =1.5pt,decoration={markings, mark=at position 1 with {\arrow{>}}},postaction={decorate}](1.5,2.5)--(1.5,0);
		\draw[line width =1pt] (-1.5,0.5)--(-1,0.5);
		\draw[line width =1pt] (-0.5,0.5)--(1,0.5);
		\draw[line width =1pt] (-1.5,1)--(-1,1);
		\draw[line width =1pt] (-0.5,1.5)--(1,1.5);
		\draw[line width =1pt] (-1.5,2)--(-1,2);
		\draw[line width =1pt] (-0.5,2)--(1,2);
		\draw [fill=gray!20] (0.3,0.5) circle (0.2);
		\draw [fill=gray!20] (0.3,1.5) circle (0.2);
		\draw [fill=gray!20] (0.3,2) circle (0.2);
%		\draw [fill=gray!20] (0.7,0.3) rectangle (1.1, 0.7);
%		\draw [fill=gray!20] (0.7,1.3) rectangle (1.1, 1.7);
%		\draw [fill=gray!20] (0.7,1.8) rectangle (1.1, 2.2);
		%			\draw [line width =1pt] (-1,0.2) rectangle (-0.5, 2.3);
		%			\draw [line width =1pt] (-1,0.2) rectangle (-0.5, 2.3);
		\node [left] at(0,1) {$\vdots$};
		\node [left] at(-1,1.5) {$\vdots$};
		\node  at(0.3,0.5) {\small $N$};
%		\node  at(0.9,0.5) {\small $\bar{H}$};
%		\node  at(0.9,2) {\small $\bar{H}$};
%		\node  at(0.9,1.5) {\small $\bar{H}$};
		\node  at(0.3,1.5) {\small $N$};
		\node  at(0.3,2) {\small $N$};
		\node  at(-0.75,1.25) {$\bar{H}$};
	\end{tikzpicture}}$
to denote the picture obtained from parallel strands by first taking the negative half-twist and then taking $N$ parallel copies of each strand. We use $\raisebox{-.15in}{
	
	\begin{tikzpicture}%[dline /. style ={line width =2pt}]
		\tikzset{->-/.style=
			
			{decoration={markings,mark=at position #1 with
					
					{\arrow{latex}}},postaction={decorate}}}

		\filldraw[draw=white,fill=gray!20] (-1.5,0.1) rectangle (0, 0.9);
%		\draw [line width =1.5pt,decoration={markings, mark=at position 1 with {\arrow{>}}},postaction={decorate}](0,0.9)--(0,0.1);
		\draw[line width =1pt] (-1.5,0.5)--(0,0.5);
		\draw [fill=gray!20] (-1,0.5) circle (0.2);
%		\node [right] at(0,0.5) {$i$};
		\node  at(-1,0.5) {\small $N$};
\draw [fill=gray!20] (-0.6,0.25) rectangle (-0.2, 0.75);
\node  at(-0.4,0.5) {\small $\bar H$};
	\end{tikzpicture}
	}$ to denote the picture obtained from $N$ parallel strands by taking the negative half-twist.

	We have 
	\begin{align*}
	h_{e}\circ \cF\left(
	\raisebox{-.40in}{
		\begin{tikzpicture}%[dline /. style ={line width =2pt}]
			\tikzset{->-/.style=
				{decoration={markings,mark=at position #1 with
						{\arrow{latex}}},postaction={decorate}}}
			\filldraw[draw=white,fill=gray!20] (-1.5,0) rectangle (0, 2.5);
			\draw [line width =1.5pt,decoration={markings, mark=at position 1 with {\arrow{>}}},postaction={decorate}](0,2.5)--(0,0);
			\draw[line width =1pt] (-1.5,0.5)--(0,0.5);
			\draw[line width =1pt] (-1.5,1)--(0,1);
			\draw[line width =1pt] (-1.5,2)--(0,2);
			\node [left] at(0,1.5) {$\vdots$};
			\node [right] at(0,0.5) {$i_1$};
			\node [right] at(0,1) {$i_2$};
			\node [right] at(0,2) {$i_k$};
		\end{tikzpicture}
	}\right)=&
		h_{e}\left(
	\raisebox{-.40in}{
		\begin{tikzpicture}%[dline /. style ={line width =2pt}]
			\tikzset{->-/.style=
				{decoration={markings,mark=at position #1 with
						{\arrow{latex}}},postaction={decorate}}}
			\filldraw[draw=white,fill=gray!20] (-1.5,0) rectangle (0, 2.5);
			\draw [line width =1.5pt,decoration={markings, mark=at position 1 with {\arrow{>}}},postaction={decorate}](0,2.5)--(0,0);
			\draw[line width =1pt] (-1.5,0.5)--(0,0.5);
			\draw[line width =1pt] (-1.5,1)--(0,1);
			\draw[line width =1pt] (-1.5,2)--(0,2);
			\draw [fill=gray!20] (-1,0.5) circle (0.2);
			\draw [fill=gray!20] (-1,1) circle (0.2);
			\draw [fill=gray!20] (-1,2) circle (0.2);
			\node [left] at(0,1.5) {$\vdots$};
			\node [right] at(0,0.5) {$i_1$};
			\node [right] at(0,1) {$i_2$};
			\node [right] at(0,2) {$i_k$};
			\node  at(-1,0.5) {\small $N$};
			\node  at(-1,1) {\small $N$};
			\node  at(-1,2) {\small $N$};
		\end{tikzpicture}
	}\right)
	\\=&
	 \left(\frac 1{ \prod_{j=1}^k c_{{i_j}}(\q)^N}\right) \cdot
	\raisebox{-.40in}{		
		\begin{tikzpicture}%[dline /. style ={line width =2pt}]
			\tikzset{->-/.style=				
				{decoration={markings,mark=at position #1 with
						{\arrow{latex}}},postaction={decorate}}}
			\filldraw[draw=white,fill=gray!20] (-1.5,0) rectangle (1.5, 2.5);
			\draw [line width =1pt] (-1,0.2) rectangle (-0.5, 2.3);
			\draw [line width =1.5pt,decoration={markings, mark=at position 1 with {\arrow{>}}},postaction={decorate}](1.5,2.5)--(1.5,0);
			\draw[line width =1pt] (-1.5,0.5)--(-1,0.5);
			\draw[line width =1pt] (-0.5,0.5)--(1.5,0.5);
			\draw[line width =1pt] (-1.5,1)--(-1,1);
			\draw[line width =1pt] (-0.5,1.5)--(1.5,1.5);
			\draw[line width =1pt] (-1.5,2)--(-1,2);
			\draw[line width =1pt] (-0.5,2)--(1.5,2);
			\draw [fill=gray!20] (0.3,0.5) circle (0.2);
			\draw [fill=gray!20] (0.3,1.5) circle (0.2);
			\draw [fill=gray!20] (0.3,2) circle (0.2);
			\draw [fill=gray!20] (0.7,0.3) rectangle (1.1, 0.7);
			\draw [fill=gray!20] (0.7,1.3) rectangle (1.1, 1.7);
			\draw [fill=gray!20] (0.7,1.8) rectangle (1.1, 2.2);
%			\draw [line width =1pt] (-1,0.2) rectangle (-0.5, 2.3);
%			\draw [line width =1pt] (-1,0.2) rectangle (-0.5, 2.3);
			\node [left] at(0,1) {$\vdots$};
			\node [left] at(-1,1.5) {$\vdots$};
			\node [right] at(1.5,0.5) {$\overline{i_k}$};
			\node [right] at(1.5,1.5) {$\overline{i_2}$};
			\node [right] at(1.5,2) {$\overline{i_1}$};
			\node  at(0.3,0.5) {\small $N$};
			\node  at(0.9,0.5) {\small $\bar{H}$};
			\node  at(0.9,2) {\small $\bar{H}$};
			\node  at(0.9,1.5) {\small $\bar{H}$};
			\node  at(0.3,1.5) {\small $N$};
			\node  at(0.3,2) {\small $N$};
			\node  at(-0.75,1.25) {$\bar{H}$};
		\end{tikzpicture}
	}\\=&
	\left(\frac 1{ \prod_{j=1}^k c_{{i_j}}(1)}\right) \cdot
	\raisebox{-.40in}{		
	\begin{tikzpicture}%[dline /. style ={line width =2pt}]
		\tikzset{->-/.style=				
			{decoration={markings,mark=at position #1 with
					{\arrow{latex}}},postaction={decorate}}}
		\filldraw[draw=white,fill=gray!20] (-1.5,0) rectangle (1.5, 2.5);
		\draw [line width =1pt] (-1,0.2) rectangle (-0.5, 2.3);
		\draw [line width =1.5pt,decoration={markings, mark=at position 1 with {\arrow{>}}},postaction={decorate}](1.5,2.5)--(1.5,0);
		\draw[line width =1pt] (-1.5,0.5)--(-1,0.5);
		\draw[line width =1pt] (-0.5,0.5)--(1.5,0.5);
		\draw[line width =1pt] (-1.5,1)--(-1,1);
		\draw[line width =1pt] (-0.5,1.5)--(1.5,1.5);
		\draw[line width =1pt] (-1.5,2)--(-1,2);
		\draw[line width =1pt] (-0.5,2)--(1.5,2);
		\draw [fill=gray!20] (0.3,0.5) circle (0.2);
		\draw [fill=gray!20] (0.3,1.5) circle (0.2);
		\draw [fill=gray!20] (0.3,2) circle (0.2);
%		\draw [fill=gray!20] (0.7,0.3) rectangle (1.1, 0.7);
%		\draw [fill=gray!20] (0.7,1.3) rectangle (1.1, 1.7);
%		\draw [fill=gray!20] (0.7,1.8) rectangle (1.1, 2.2);
		%			\draw [line width =1pt] (-1,0.2) rectangle (-0.5, 2.3);
		%			\draw [line width =1pt] (-1,0.2) rectangle (-0.5, 2.3);
		\node [left] at(0,1) {$\vdots$};
		\node [left] at(-1,1.5) {$\vdots$};
		\node [right] at(1.5,0.5) {$\overline{i_k}$};
		\node [right] at(1.5,1.5) {$\overline{i_2}$};
		\node [right] at(1.5,2) {$\overline{i_1}$};
		\node  at(0.3,0.5) {\small $N$};
%		\node  at(0.9,0.5) {\small $\bar{H}$};
%		\node  at(0.9,2) {\small $\bar{H}$};
%		\node  at(0.9,1.5) {\small $\bar{H}$};
		\node  at(0.3,1.5) {\small $N$};
		\node  at(0.3,2) {\small $N$};
		\node  at(-0.75,1.25) {$\bar{H}$};
	\end{tikzpicture}},
%	}\\=&\cF\circ h_e\left(
%	\raisebox{-.40in}{
%%		
%		\begin{tikzpicture}%[dline /. style ={line width =2pt}]
%			\tikzset{->-/.style=
%%				
%				{decoration={markings,mark=at position #1 with
%%						
%						{\arrow{latex}}},postaction={decorate}}}
%%			
%			\filldraw[draw=white,fill=gray!20] (-1.5,0) rectangle (0, 2.5);
%			\draw [line width =1.5pt,decoration={markings, mark=at position 1 with {\arrow{>}}},postaction={decorate}](0,2.5)--(0,0);
%			\draw[line width =1pt] (-1.5,0.5)--(0,0.5);
%			\draw[line width =1pt] (-1.5,1)--(0,1);
%			\draw[line width =1pt] (-1.5,2)--(0,2);
%			\node [left] at(0,1.5) {$\vdots$};
%			\node [right] at(0,0.5) {$i_1$};
%			\node [right] at(0,1) {$i_2$};
%			\node [right] at(0,2) {$i_k$};
%		\end{tikzpicture}
%	}\right),
	\end{align*}
	where the last equality is because of relations \eqref{cross} and \eqref{arc}.
	\begin{align*}
		\cF\circ h_e\left(
		\raisebox{-.40in}{
			\begin{tikzpicture}%[dline /. style ={line width =2pt}]
				\tikzset{->-/.style=
					{decoration={markings,mark=at position #1 with
							{\arrow{latex}}},postaction={decorate}}}
				\filldraw[draw=white,fill=gray!20] (-1.5,0) rectangle (0, 2.5);
				\draw [line width =1.5pt,decoration={markings, mark=at position 1 with {\arrow{>}}},postaction={decorate}](0,2.5)--(0,0);
				\draw[line width =1pt] (-1.5,0.5)--(0,0.5);
				\draw[line width =1pt] (-1.5,1)--(0,1);
				\draw[line width =1pt] (-1.5,2)--(0,2);
				\node [left] at(0,1.5) {$\vdots$};
				\node [right] at(0,0.5) {$i_1$};
				\node [right] at(0,1) {$i_2$};
				\node [right] at(0,2) {$i_k$};
			\end{tikzpicture}
		}\right)=&
		\left(\frac 1{ \prod_{j=1}^k c_{{i_j}}(1)}\right) \cdot \cF\left(
		\raisebox{-.40in}{
			\begin{tikzpicture}%[dline /. style ={line width =2pt}]
				\tikzset{->-/.style=
					{decoration={markings,mark=at position #1 with
							{\arrow{latex}}},postaction={decorate}}}
				\filldraw[draw=white,fill=gray!20] (-1.5,0) rectangle (0, 2.5);
				\draw [line width =1pt] (-1,0.2) rectangle (-0.5, 2.3);
				\draw [line width =1.5pt,decoration={markings, mark=at position 1 with {\arrow{>}}},postaction={decorate}](0,2.5)--(0,0);
				\draw[line width =1pt] (-1.5,0.5)--(-1,0.5);
				\draw[line width =1pt] (-0.5,0.5)--(0,0.5);
				\draw[line width =1pt] (-1.5,1)--(-1,1);
				\draw[line width =1pt] (-0.5,1.5)--(0,1.5);
				\draw[line width =1pt] (-1.5,2)--(-1,2);
				\draw[line width =1pt] (-0.5,2)--(0,2);
				\node [left] at(0,1) {$\vdots$};
				\node [left] at(-1,1.5) {$\vdots$};
				\node [right] at(0,0.5) {$\overline{i_k}$};
				\node [right] at(0,1.5) {$\overline{i_2}$};
				\node [right] at(0,2) {$\overline{i_1}$};
				\node  at(-0.75,1.25) {$\bar{H}$};
			\end{tikzpicture}
		}\right)
		\\=&
		\left(\frac 1{ \prod_{j=1}^k c_{{i_j}}(1)}\right) \cdot
		\raisebox{-.40in}{		
			\begin{tikzpicture}%[dline /. style ={line width =2pt}]
				\tikzset{->-/.style=				
					{decoration={markings,mark=at position #1 with
							{\arrow{latex}}},postaction={decorate}}}
				\filldraw[draw=white,fill=gray!20] (-1.5,0) rectangle (1.5, 2.5);
				\draw [line width =1pt] (-1,0.2) rectangle (-0.5, 2.3);
				\draw [line width =1.5pt,decoration={markings, mark=at position 1 with {\arrow{>}}},postaction={decorate}](1.5,2.5)--(1.5,0);
				\draw[line width =1pt] (-1.5,0.5)--(-1,0.5);
				\draw[line width =1pt] (-0.5,0.5)--(1.5,0.5);
				\draw[line width =1pt] (-1.5,1)--(-1,1);
				\draw[line width =1pt] (-0.5,1.5)--(1.5,1.5);
				\draw[line width =1pt] (-1.5,2)--(-1,2);
				\draw[line width =1pt] (-0.5,2)--(1.5,2);
				\draw [fill=gray!20] (0.3,0.5) circle (0.2);
				\draw [fill=gray!20] (0.3,1.5) circle (0.2);
				\draw [fill=gray!20] (0.3,2) circle (0.2);
				%		\draw [fill=gray!20] (0.7,0.3) rectangle (1.1, 0.7);
				%		\draw [fill=gray!20] (0.7,1.3) rectangle (1.1, 1.7);
				%		\draw [fill=gray!20] (0.7,1.8) rectangle (1.1, 2.2);
				%			\draw [line width =1pt] (-1,0.2) rectangle (-0.5, 2.3);
				%			\draw [line width =1pt] (-1,0.2) rectangle (-0.5, 2.3);
				\node [left] at(0,1) {$\vdots$};
				\node [left] at(-1,1.5) {$\vdots$};
				\node [right] at(1.5,0.5) {$\overline{i_k}$};
				\node [right] at(1.5,1.5) {$\overline{i_2}$};
				\node [right] at(1.5,2) {$\overline{i_1}$};
				\node  at(0.3,0.5) {\small $N$};
				%		\node  at(0.9,0.5) {\small $\bar{H}$};
				%		\node  at(0.9,2) {\small $\bar{H}$};
				%		\node  at(0.9,1.5) {\small $\bar{H}$};
				\node  at(0.3,1.5) {\small $N$};
				\node  at(0.3,2) {\small $N$};
				\node  at(-0.75,1.25) {$\bar{H}$};
			\end{tikzpicture}
		}.
	\end{align*}
	
	Condition (3): The proof is similar with the proof for Condition (2).
\end{proof}

%Lemma \ref{change} implies  $h_e: \sS_1(M,\cN)\rightarrow \sS_1\MN$  induces a bijection
%	\begin{align*} (h_e)^{*}:\text{MaxSpec}(\sS_1\MN)\rightarrow \text{MaxSpec}(\sS_1(M,\cN)),\;%
%			\rho\mapsto \rho\circ h_e,%
%		\end{align*}
%	where $\rho \in \text{MaxSpec}(\sS_1\MN)$ is regarded as an algebra homomorphism. % from $\sS_1\MN$ to $\mathbb{C}$.

For any $\q$ and $\rho\in \text{MaxSpec}(\sS_1\MN)$, Lemma \ref{con} implies $$h_e : \S\MN\rightarrow \S\MN$$ induces the linear isomorphism
$$(h_e)_{\rho} : \S\MN_{(h_e)^{*}(\rho)}\rightarrow \S\MN_{\rho}.$$

Let $\MN$ be a marked three maniold,  let $D$ be a properly embedded  disk in $M$, and let $u$ be an embedded oriented open interval in $D$. Then $\text{Cut}_{(D,u)}\MN$ has two copies of $u$, denoted them as $u_1,u_2$. Suppose $\cN^{'}$ is obtained from $\cN$ by adding one extra marking $e$ such that $e$ and $\partial D$ belong to the same  component of $\partial M$. We use $l:\MN\rightarrow (M,\cN^{'})$ to denote the obvious embedding. We have $(M,\cN^{'})$ is isomorphic to $(\text{Cut}_{(D,u)}\MN)_{u_1\bT u_2}$.

\begin{lemma}\label{key}
	There exists an isomorphism $\varphi:(M,\cN^{'})\rightarrow (\text{Cut}_{(D,u)}\MN)_{u_1\bT u_2}$ such that
	$$QF_{u_1,u_2} \circ h_{u_2}\circ \Theta_{(D,u)}
	= \varphi_{*}\circ l_{*}.$$
\end{lemma}
\begin{proof}
	Proposition 6.3 in \cite{wang2023stated}.
\end{proof}

% Suppose $D$ is a properly embedded disk and $\beta$ is an oriented open interval contained in $D$. We know there is a linear map
%$\Theta:\S\MN\rightarrow \S(\text{Cut}_{(D,\beta)}(M,N),v)$.  $\text{Cut}_{(D,\beta)}(M,N)$ has two markings labeled by $\beta_1,\beta_2$ as in Figure \ref{fig:1}. Then 
%$(M,N^{'})\simeq(\text{Cut}_{(D,\beta)}(M,N))_{\beta_1\Delta \beta_2}$ where $N^{'}$ is obtained from $N$ by adding one extra marking. 
%We use $\varphi$ to denote this obvious isomorphism, and use $\varphi_{*}$ to denote the isomorphism 
%from $S_n((M,N^{'}),v)$ to $S_n((\text{Cut}_{(D,\beta)})_{\beta_1\Delta \beta_2},v)$ induced by $\varphi$.
%Recall that there is linear map $l_{ad}: S_n(M,N,v)\rightarrow S_n(M,N^{'},v)$ induced by the embedding $(M,N)\rightarrow (M,N^{'})$.

\begin{theorem}
	Let $\MN$ be a marked three maniold,  let $D$ be a properly embedded  disk in $M$, and let $u$ be an embedded oriented open interval in $D$. Suppose the component $V$ of $\partial M$  contains  $\partial D$ and $V\cap \cN\neq\emptyset$. Then, for any $\rho\in \text{MaxSpec}(\sS_1(\text{Cut}_{(D,u)}\MN))$, we have
	$$\Theta_{\rho}:\S\MN_{\Theta^{*}(\rho)}\rightarrow \S(\text{Cut}_{(D,u)}\MN)_{\rho}$$
	is injective.
\end{theorem}
\begin{proof}
	Here we use the notations in Lemma \ref{key}. Suppose $$\rho^{'}\in \text{MaxSpec}(\sS_1((\text{Cut}_{(D,u)}\MN)_{u_1\bT u_2})).$$
	Lemmas \ref{embedding}, \ref{splitting}, \ref{QF}  and \ref{key} imply
	$$(QF_{u_1,u_2})_{\rho^{'}} \circ (h_{u_2})_{QF^{*}(\rho^{'})}\circ (\Theta_{(D,u)})_{(h_{u_2})^{*}(QF^{*}(\rho^{'}))}
	= \varphi_{\rho^{'}}\circ l_{\varphi^{*}(\rho)}.$$
	Since $(QF_{u_1,u_2})_{\rho^{'}},\, (h_{u_2})_{QF^{*}(\rho^{'})}\, , \varphi_{\rho^{'}}$ are all linear isomorphisms, we have $$\text{Ker}( (\Theta_{(D,u)})_{(h_{u_2})^{*}(QF^{*}(\rho^{'}))}) =\text{Ker}(l_{\varphi^{*}(\rho)}).$$
	Lemma \ref{key2} implies $\text{Ker}(\Theta_{(D,u)})_{(h_{u_2})^{*}(QF^{*}(\rho^{'}))}) =\text{Ker}(l_{\varphi^{*}(\rho)}) = 0$. This completes the proof because  both $(h_{u_2})^{*}$ and $QF^{*}$ are bijections.
\end{proof}

\begin{remark}
	The Frobenius map for $SL_n$ was constructed in \cite{wang2023stated} when every component of $M$ contains at least one marking. All the parallel results in this section can be easily generalized to the $SL_n$ case when every component of $M$ contains at least one marking.
\end{remark}

{
\begin{remark}
When the marking set is empty, we expect the splitting map for the representation-reduced stated module is injective if $M$ is the thickening of an oriented surface. 
\end{remark}}

\section{Geometric stated skein}\label{5}
Let $M$ be an oriented connected closed three manifold. Recall that, for any positive integer $k$, we use 
$M_{k}$ to denote the marked three manifold obtained from $M$ by removing $k$ open three dimensional balls and adding one marking to each newly created sphere boundary component. Then $M_k$ is defined up to isomorphism. In this section, we will show $\text{dim}_{\mathbb{C}}
\S( M_{k})_{\rho} = 1$ for any $\rho\in\text{MaxSpec}(\sS_1(M_k))$.

\def \sgk {\Sigma}
\def \hgk {H}
\def \pgk {\partial_+(H)}
\def \mgk {\partial_{-}(H)}

We use $H_g$ to denote the genus $g$ handlebody. For any positive integer $k$, we use $H_{g,k}$ to denote the marked three manifold obtained from $H_g$ by adding $k$ markings on $\partial H_g$. Then $H_{g,k}$ is defined up to isomorphism. We use $\cN_k$ to denote the union of all the markings in $H_{g,k}$, and use $\overline{\cN_k}$ to denote the closure of $\cN_{k}.$
To simplify the notation, we will use $H$ to denote $H_{g,k}$ in this section.

For each marking $e$ of $H$, we can regard its closure $\bar e$ as an embedding from $[0,1]$ to $\partial H_g$.
We choose an embedded disk $D_e$ (respectively $D^{'}_e$) in $\partial H_g$ such that $\bar e\subset D_e$ (respectively $\bar e\subset D^{'}_e$) and  $\bar e\cap \partial D_e = \bar e(0)$ (respectively $\bar e\cap \partial D_e^{'} =\bar e(1)$). We also require there is no intersection among $D_e$ (or $D_e^{'}$).  Then  $\mgk$ (respectively $\pgk$) is obtained from $\partial H_g$ by removing the interior of all $D_e$
(respectively $D_e^{'}$), please refer to Figures \ref{fg1} and \ref{fg2}. 
$\mgk$ (respectively $\pgk$) is  a marked surface with marked points $\mgk \cap \overline{\cN_k}$ (respectively $\pgk \cap \overline{\cN_k}$).

 Then
$\S(\mgk)$ has a right action on $\S(\hgk)$, and
$\S(\pgk)$ has a left action on $\S(\hgk)$.

\begin{figure}[h]
	\centering
	\includegraphics[scale=0.5]{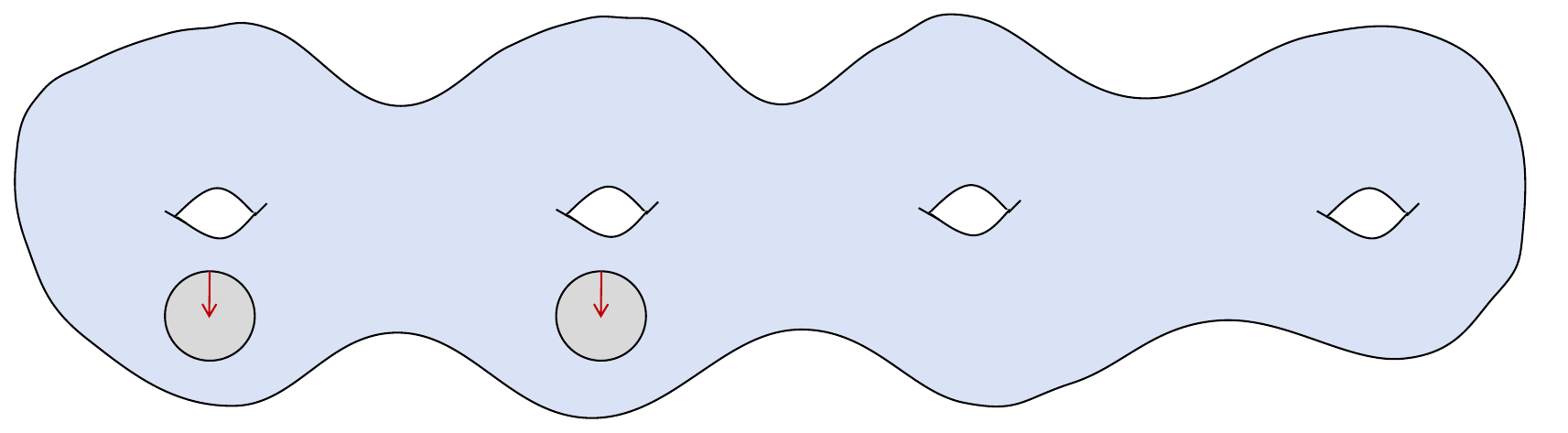}
	\caption{The picture is for $H$, where $g=4$ and $k=2$. 
The gray disks are embedded disks in the boundary of the handlebody. The red arrows contained in the gray disks are markings.
The sub-surface of $\partial H_g$ colored by  blue is denoted as $\mgk$ (as a surface, $\mgk$ is obtained from $\partial H_g$ by removing $k$ open disks). It is  a marked surface with marked points $\mgk \cap \overline{\cN_k}$.}\label{fg1}
\end{figure}

\begin{figure}[h]
	\centering
	\includegraphics[scale=0.5]{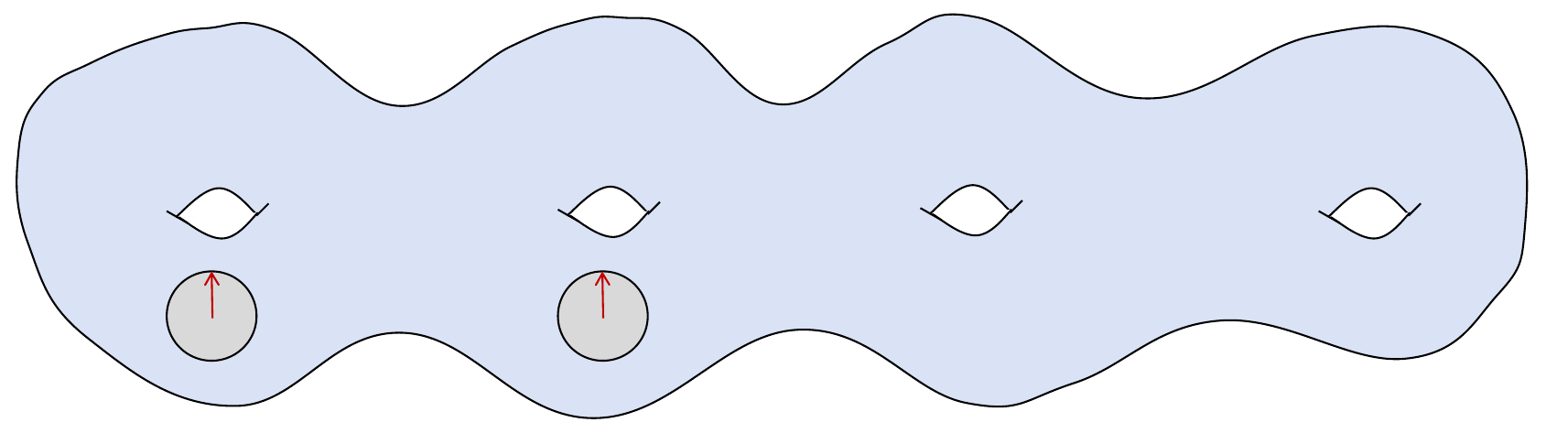}
	\caption{The picture is for $H$, where $g=4$ and $k=2$. The gray disks are embedded disks in the boundary of the handlebody. The red arrows contained in the gray disks are markings.
 The sub-surface of $\partial H_g$ colored by  blue is denoted as $\pgk$ (as a surface, $\pgk$ is obtained from $\partial H_g$ by removing $k$ open disks). It is  a marked surface with marked points $\pgk \cap \overline{\cN_k}$.}\label{fg2}
\end{figure}

We use $\Sigma_g$ to denote the closed surface of genus $g$. For any positive integer $k$, we use $\Sigma_{g,k}$ to denote the marked surface  obtained from $\Sigma_g$ by removing $k$ open disks and equipping each newly created boundary component with one marked point.  
To simplify the notation, we will use $\Sigma$ to denote $\Sigma_{g,k}$ in this section.
We have $\Sigma\simeq \partial_+(H)\simeq \partial_{-}(H)$.
Then any isomorphism $f:\sgk\rightarrow \pgk$ induces a left action of $\S(\sgk)$  on $\S(\hgk)$. Similarly, any isomorphism $g:\sgk\rightarrow \mgk$ induces a right action of $\S(\sgk)$  on $\S(\hgk)$.

We know there is an injective algebra homomorphism
$\cF:\sS_1(\sgk)\rightarrow \S(\sgk)$. From \cite{yu2023center}, we know 
$\text{Im}\cF$ is the center of $\S(\sgk)$. 

Let $\theta: \S(\sgk)\rightarrow \text{End}(V)$ be an irreducible representation of $\S(\sgk)$. Then $\theta$ induces an algebra homomorphism $\varepsilon_{\theta}:\sS_1(\sgk)\rightarrow \bC$ such that 
$\varepsilon_{\theta}(x)Id_V = \theta(\cF(x))$ for any $x\in\sS_1(\sgk)$. We will call $\varepsilon_{\theta}$ the {\bf classical shadow} of $\theta$.
Then the Azumaya locus $\mathcal{A}(\S(\sgk))$ of $\S(\sgk)$ is defined to be 
\begin{align*}
&\mathcal{A}(\S(\sgk)) = \{\rho\in\text{MaxSpec}(\sS_1(\sgk))\mid
 \text{there exists a unique}\\&\text{ irreducible representation } \theta\text{ of }\S(\sgk)\text{ such that } \varepsilon_{\theta} = \rho\}.
\end{align*}
An irreducible representation of $\S(\sgk)$ is called an {\bf Azumaya  representation} if its classical shadow lives in the Azumaya locus of $\S(\sgk)$.

Since $\text{Im}\cF$ is the center of $\S(\sgk)$, the PI-dimension of $\S(\sgk)$ is equal to the square root of the rank of $\S(\sgk)$ over $\sS_1(\sgk)$. 
 From  \cite{wang2023finiteness,yu2023center}, we know the PI-dimension of $\S(\sgk)$ is $N^{3(g+k-1)}$. For any positive integer $n$, we use 
 $\text{Mat}_n(\bC)$ to denote the algebra of all $n$ by $n$ complex matrices. Then 
 \begin{align*}
 	\mathcal{A}(\S(\sgk)) = &\{\rho\in\text{MaxSpec}(\sS_1(\sgk))\mid\\
 	&\S(\sgk)_{\rho}\simeq\text{Mat}_{n}(\bC),\text{ where }n = N^{3(g+k-1)}\}.
 \end{align*}
 
 \def \M {\text{MaxSpec}} 
 
 Suppose $\rho\in\M(\sS_1(\hgk))$ and $f$ is an isomorphism from $\sgk$ to $ \pgk$. It is easy to show that the left action of $\S(\sgk)$ (induced by $f$) on $\S(\hgk)$ reduces to a left action of $\S(\sgk)$ on $\S(\hgk)_{\rho}$.
 
 We can regard $\pgk\times [0,1]$ as a closed regular neighborhood of $\pgk$.
  We use $L$ to denote the embedding from $\pgk\times [0,1]$ to $\hgk$. Then $L$ induces an algebra homomorphism $L_{*}:\sS_1(\pgk)\rightarrow \sS_1(\hgk)$. Then $L_*$ induces a map $L^{*}:\M(\sS_1(\hgk))\rightarrow \M(\sS_1(\pgk))$.
  
  Any isomorphism $f:\sgk\rightarrow \pgk$ induces an isomorphism 
  from $\sgk\times [0,1]$ to $ \pgk\times [0,1]$, which is still denoted as $f$.
  
  In the following theorem, we prove 
  $\S(\hgk)_{\rho}$ is an Azumaya representation for $\S(\sgk)$.
   This generalizes Theorem 12.1 in \cite{frohman2023sliced} (they prove a parallel result for the non-stated case) to the stated case.

 For any marked point $p$ of $\sgk$, we define a curve $\alpha(p)$ as in Figure \ref{fg3}.
 
 \begin{figure}[h]
 	\centering
 	\includegraphics[scale=0.5]{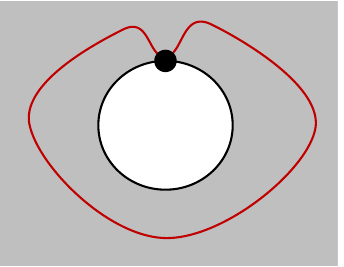}
 	\caption{The circle is one of the boundary components of $\sgk$, the black dot is the marked point $p$ on this boundary component. The red curve is $\alpha(p)$.}\label{fg3}
 \end{figure}
  
 \begin{theorem}\label{Azu}
 	Let $k$ be a positive integer,  let $\rho\in\M(\sS_1(\hgk))$, and let $f$ be an isomorphism from $\sgk$ to $ \pgk$. Then we have the followings:

(a)  $L^*(\rho)\in\mathcal{A}(\S(\pgk))$.

(b)  $\S(\hgk)_{\rho}$ is an irreducible representation of $\S(\sgk)$ whose classical shadow is $f^*(L^*(\rho))$. Meanwhile, this irreducible representation is an Azumaya representation.
 	
 \end{theorem}
 \begin{proof}
(a) For each marked point $p$ of $\sgk$, we have
 	$L(f(\alpha(p)))$ is trivial in $H_{g}$ (that is, it bounds an embedded disk in $H_g$). Then, for any $\rho\in\M(\sS_1(\hgk))$,
 	Theorem 8.1 in \cite{karuo2023classification} shows $$L^*(\rho)\in\mathcal{A}(\S(\pgk)).$$

 (b) From the above discussion, we know $\S(\sgk)$ has a left action on $\S(\hgk)_{\rho}$ (induced by $f$). For any stated tangle $\alpha$ in $\sgk\times [0,1]$ and any stated tangle $\beta$ in $\hgk$, we have
 	\begin{align*} 
 	\cF(\alpha)\cdot (\beta\otimes 1 )= &(L_*(f_*(\cF(\alpha)))\cup \beta)\otimes 1 = (\cF(L_*(f_*(\alpha)))\cup\beta)\otimes 1
 	\\ = &(L_*(f_*(\alpha))\cdot\beta)\otimes 1
 	= \beta\otimes L_*(f_*(\alpha))\cdot 1 = \rho(L_*(f_*(\alpha)))\beta\otimes 1,
 	 \end{align*}
 	where we regard $\alpha$ as an element in $\sS_1(\sgk)$  and regard $\beta$ as an element in $\S(\hgk)$. Then the classical shadow of the representation $\S(\hgk)_{\rho}$ (as a representation of $\S(\sgk)$) is $f^*(L^*(\rho))$.
 	Thus this representation reduces to a representation of $\S(\sgk)_{f^*(L^*(\rho))}$. We have $\S(\sgk)_{f^*(L^*(\rho))}\simeq \text{Mat}_{N^{3(g+k-1)}}(\bC)$ because $L^*(\rho)\in\mathcal{A}(\S(\pgk))$
 	and $f$ is an isomorphism. From Theorem 4.10 in \cite{wang2023frobenius}, we know $\text{dim}_{\bC} \S(\hgk)_{\rho} = N^{3(g+k-1)}$. Then $\S(\hgk)_{\rho}$ is an irreducible representation of $\S(\sgk)_{f^*(L^*(\rho))}$. This completes the proof.
 \end{proof}

 %A parallel result for the non-stated case (that is the skein case) for Theorem \ref{Azu} is proved in \cite{frohman2023sliced}.

To distinguish two copies of $H$. We use $H^{+}$ and $H^{-}$ to denote them (that is $H^{+} = H^{-} = H$, but we denote these two copies with different notations).
We also use $\partial H^{+}$ to denote $\pgk$ and use $\partial H^{-}$ to denote $\mgk$.

We suppose $H$ has a fixed orientation, and the orientations of $\pgk$ and $\mgk$ are inherited from $H$.
Let $f$ be an isomorphism from $\sgk$ to $\pgk$, and let $g$ be an isomorphism from $\sgk$ to $\mgk$. Then $f\circ g^{-1}$ is an isomorphism from $\partial H^{-}$ to $\partial H^{+}$ (here we require $f\circ g^{-1}$ is an orientation reversing isomorphism).

We use $H^{-}\cup_{f\circ g^{-1}} H^{+}$ to denote $(H^{-}\cup H^{+})/(x = f(g^{-1}(x)),x\in \partial H^{-})$. Then 
$H^{-}\cup_{f\circ g^{-1}} H^{+}$ is a marked three manifold. Actually, there exists an oriented closed three manifold $M$ such that $M_k$ is isomorhic to $H^{-}\cup_{f\circ g^{-1}} H^{+}$. Meanwhile, for any oriented connected closed three manifold  $M$ and any positive integer $k$, there exist a non-negative integer $g$ and isomorphisms 
$f:\sgk\rightarrow \pgk = \partial H^{+} ,\;g:\sgk\rightarrow \mgk=\partial H^{-}$ such that $H^{-}\cup_{f\circ g^{-1}} H^{+}$ is isomorphic to $M_k$.

We use $L_+$ (respectively $L_{-}$) to denote the obvious embedding from $H^{+}$ (respectively $H^{-}$) to $H^{-}\cup_{f\circ g^{-1}} H^{+}$. Then we have $L_+\circ f = L_- \circ g,$ which is an 
embedding from $\sgk$ to $H^{-}\cup_{f\circ g^{-1}} H^{+}$.
Then $L_+\circ f$ induces an algebra homomorphism 
$$(L_+\circ f)_* = (L_+)_{*}\circ f_*:\sS_1(\sgk)\rightarrow \sS_1(H^{-}\cup_{f\circ g^{-1}} H^{+}).$$
It further induces a map
$$(L_+\circ f)^* = f^{*}\circ(L_+)^{*}:\M(\sS_1(H^{-}\cup_{f\circ g^{-1}} H^{+}))\rightarrow\M(\sS_1(\sgk)).$$
 %and we use 
%$J$ to denote this embedding (from $\sgk$ to $H^{+}\cup_{f\circ g^{-1}} H^{-}$).

 Suppose $R$ is a commutative unital ring, $U$ and $V$ are two $R$-modules, and $I$ is an ideal of $R$. We state the following results for classical module theory without giving proofs.
\begin{align}\label{eq_U}
    \text{We have } U\otimes_{R} R/I\simeq \frac{U}{I\cdot U}\text{ as }  R \text{-modules}.
\end{align}
 \begin{align}\label{eq_UV}
    \text{We have } U\otimes_{R} V\simeq \frac{U}{I\cdot U}\otimes_{R/I} V\text{ as }  R \text{-modules if }I\cdot V=0.
\end{align}

\begin{theorem}\label{main}
	Let $k$ be a positive integr, let $f$ (respectively $g$) be an isomorphism from $\sgk$ to 
	$\partial H^{+}$ (respectively $\partial H^{-}$) such that $f\circ g^{-1}$ is orientation reversing isomorphism, and let $\rho\in \M(\sS_1(H^{-}\cup_{f\circ g^{-1}} H^{+}))$. Then we have the followings:

(a) $(L_+\circ f)^*(\rho)\in\mathcal{A}(\S(\sgk))$

(b)	$\text{dim}_{\mathbb{C}} \S(H^{-}\cup_{f\circ g^{-1}} H^{+})_{\rho} = 1.$
	
\end{theorem}
\begin{proof}
	
(a) The proof is similar with the proof for (a) in Theorem  \ref{Azu}.
	Theorem 8.1 in \cite{karuo2023classification} shows $$(L_+\circ f)^*(\rho)\in\mathcal{A}(\S(\sgk)),$$
	for any $\rho\in \M(\sS_1(H^{-}\cup_{f\circ g^{-1}} H^{+})).$

(b) We will also  regard $f$ (respectively $g$) as an embedding from $\sgk$  to $H^{+}$ (respectively $H^{-}$).
	Here we use a technique used in page  48 in
	\cite{frohman2023sliced}. We know $\S(\sgk)$ has a left (respectively right) action  on $\S(H^{+})$ (respectively $\S(H^{-})$).  
	 From Theorem 6.5 in \cite{costantino2022stated}, we have 
	 \begin{align}\label{eq_handle}
	 \S(H^{-})\otimes_{\S(\sgk)} \S(H^{+})\simeq \S(H^{-}\cup_{f\circ g^{-1}} H^{+}).
      \end{align}
	 This isomorphism is given by sending $\alpha\otimes_{\S(\sgk)}\beta $
	 to $L_-(\alpha)\cup L_+(\beta)$, where $\alpha$ (respectively $\beta$) is a stated tangle in $H^{-}$ (respectively $H^{+}$).
	 
	 The surjective algebra homomorphism $f_{*}:\sS_1(\sgk)\rightarrow \sS_1(H^{+})$ induces an action of $\sS_1(\sgk)$ on $\S(H^{+})$. For any $\rho^{'}\in\text{MaxSpec}(\sS_1(H^{+}))$, we  have
\begin{align}\label{iso_rho}
\text{Ker}(\rho^{'})\cdot \S(H^{+}) = \text{Ker}(\rho^{'}\circ f_*)\cdot \S(H^{+})
\end{align}
	 because $f_{*}:\sS_1(\sgk)\rightarrow \sS_1(H^{+})$ is surjective.
	 We have the similar discussion for $H^{-}$ and $H^{-}\cup_{f\circ g^{-1}} H^{+}$. From the isomorphism between $\S(H^{-})\otimes_{\S(\sgk)} \S(H^{+})$ and $ \S(H^{-}\cup_{f\circ g^{-1}} H^{+})$, we have $\sS_1(\sgk)$ also acts on $\S(H^{-})\otimes_{\S(\sgk)} \S(H^{+})$. The action is given by
	 \begin{equation}\label{mod}
	  \alpha\cdot (\beta_1\otimes_{\S(\sgk)}\beta_2) = \alpha\cdot\beta_1\otimes_{\S(\sgk)}\beta_2
	 =\beta_1\otimes_{\S(\sgk)}\alpha\cdot \beta_2,
	 \end{equation}
	 where $\alpha\in\sS_1(\sgk),\, \beta_1\in \S(H^{-})$ and $ \beta_2\in \S(H^{+})$. 
%	 %Then we have 
%	 $$(\S(H^{-})\otimes_{\S(\sgk)} \S(H^{+}))_{\sS_1(\sgk)}\simeq \S(H^{-})\otimes_{\S(\sgk)} (\S(H^{+}))_{\sS_1(\sgk)}$$
	 
	  Then we have 
	 \begin{align*}
	 	&\S(H^{-}\cup_{f\circ g^{-1}} H^{+})_{\rho}\simeq \frac{ \S(H^{-}\cup_{f\circ g^{-1}} H^{+})}{\text{Ker}(\rho)\cdot \S(H^{-}\cup_{f\circ g^{-1}} H^{+})}\\
	 	\simeq &\frac{
 \S(H^{-}\cup_{f\circ g^{-1}} H^{+})}{\text{Ker}((L_+\circ f)^*(\rho))\cdot \S(H^{+}\cup_{f\circ g^{-1}} H^{-})}\\
	 	\simeq &
         \frac{
	 	 \S(H^{-})\otimes_{\S(\sgk)} \S(H^{+})}{
	 	\text{Ker}((L_+\circ f)^*(\rho))\cdot \big( \S(H^{-})\otimes_{\S(\sgk)} \S(H^{+})\big )}\\
	 	\simeq &
	 	\big(\S(H^{-})\otimes_{\S(\sgk)} \S(H^{+})\big)\otimes_{\sS_1(\sgk)} \frac{ \sS_1(\sgk)}{
	 	\text{Ker}((L_+\circ f)^*(\rho))}
	 	\\\simeq% &\; (\text{this isomorphism is because of equation \eqref{mod}})
	 	&
	 	\S(H^{-})\otimes_{\S(\sgk)}\big( \S(H^{+})\otimes_{\sS_1(\sgk)}\frac{ \sS_1(\sgk)}{
	 	\text{Ker}((L_+\circ f)^*(\rho))}\big)
%	 	\\\simeq &
%	 	\S(H^{-})\otimes_{\S(\sgk)}( \S(H^{+})\otimes_{\sS_1(\sgk)} \sS_1(\sgk)/
%	 	\text{Ker}((L_+\circ f)^*(\rho)))
	 	\\\simeq &
	 	\S(H^{-})\otimes_{\S(\sgk)}\big( \frac{\S(H^{+})}{
	 	\text{Ker}((L_+\circ f)^*(\rho))\cdot  \S(H^{+})}\big)
	 	\\\simeq &
	 	\S(H^{-})\otimes_{\S(\sgk)}\big( \frac{\S(H^{+})}{
	 	\text{Ker}((L_+)^*(\rho))\cdot  \S(H^{+})}\big)
	 	\\\simeq &
	 	\S(H^{-})\otimes_{\S(\sgk)} \S(H^{+})_{(L_+)^*(\rho)}
	 	\\\simeq &
	 	\S(H^{-})_{(L_-)^*(\rho)}\otimes_{\S(\sgk)_{(L_+\circ f)^*(\rho)}} \S(H^{+})_{(L_+)^*(\rho)},
	 \end{align*}
where the first and the eighth isomorphisms come from definitions, the second and the seventh isomorphisms come from equation \eqref{iso_rho} (or a similar discussion), the third isomorphism comes from \eqref{eq_handle}, the fourth and the sixth isomorphisms come from equation \eqref{eq_U}, the fifth isomorphism comes from equation \eqref{mod}, the last isomorphism comes from equation \eqref{eq_UV}.
From (a), we know $(L_+\circ f)^*(\rho)\in \mathcal{A}(\S(\sgk))$. Then Theorem \ref{Azu} implies that
$$\S(H^{-})_{(L_-)^*(\rho)}\otimes_{\S(\sgk)_{(L_+\circ f)^*(\rho)}} \S(H^{+})_{(L_+)^*(\rho)}\simeq \mathbb{C}^{d}\otimes_{\text{Mat}_{d}(\mathbb{C})}\mathbb{C}^{d}\simeq \mathbb{C},$$
where $d = 3(g+k-1)$.
\end{proof}

%For any closed three manifold $M$ and any positive integer $k$,
%the following Theorem shows $\dim_{\bC}\S(M_k)_{\rho} = 1$.

 The following Theorem generalizes Theorem 12.2 in \cite{frohman2023sliced} (they proved a parallel result for the non-stated case) to the stated case. But, for the non-stated case, we have restrictions for $\rho$. 

\begin{theorem}\label{main2}
	Let $M$ be any oriented connected closed three manifold, and let $k$ be any positive integer. For any $\rho\in\M(\sS_1(M_{k}))$, we have $\dim_{\bC}\S(M_k)_{\rho} = 1$.
	
\end{theorem}
\begin{proof}
	Since $M_{k}$ is isomorphic to some $H^{-}\cup_{f\circ g^{-1}} H^{+}$, Theorem \ref{main} implies the Theorem.
\end{proof}

\begin{corollary}
	Suppose $n$ is a positive integer. For each $1\leq i\leq n$,
	let $M(i)$ be an oriented connected closed three manifold, and let $k_i$ be a positive integer. For any $\rho\in\M(\sS_1(\cup_{1\leq i\leq n}M(i)_{k_i}))$, we have $\dim_{\bC}\S(\cup_{1\leq i\leq n}M(i)_{k_i})_{\rho} = 1$.
\end{corollary}
\begin{proof}
	Equation \eqref{union} and Theorem \ref{main2}.
\end{proof}

\bibliography{ref.bib}

\end{document}